\documentclass[11pt,a4paper,reqno]{amsart} 
\usepackage[utf8]{inputenc}

\usepackage{lmodern}

\usepackage{a4wide}
\usepackage{fancyhdr}
\usepackage[english]{babel}

\usepackage{amsmath, mathtools}
\usepackage{amsfonts}
\usepackage{amssymb}
\usepackage{amsthm}
\usepackage[dvipsnames]{xcolor}
\usepackage{tikz}
\usepackage{pgfplots}
\usetikzlibrary{positioning}
\usetikzlibrary{intersections}
\usetikzlibrary{calc}
\usepackage{csquotes}
\usepackage{mathabx}
\usepackage{bm}
\usepackage{todonotes}
\usepackage{hyperref}
\usepackage{cleveref}

\theoremstyle{definition}

\numberwithin{equation}{section}

\title{Dynamical behaviour of a Soap Film Bridge driven by an Electrostatic Force}
\author{Lina Sophie Schmitz}
\date{August 10, 2024}
\address{Institut f\"ur Angewandte Mathematik, Leibniz Universit\"at Hannover, Welfengarten 1 \\ D-30167 Hannover, Germany}
\email{schmitz@ifam.uni-hannover.de}

\usepackage{graphicx}

\newtheoremstyle{common}
	{}    
	{}    
    {\itshape}
    {0em}
    {\bfseries}
    {}
    {.5em}
    {}
\theoremstyle{common}

\numberwithin{subsection}{section}

\numberwithin{figure}{section}

\newtheorem{thm}{{\bf Theorem}}
\numberwithin{thm}{section}

\newtheorem{lem}[thm]{{\bf Lemma}}
\newtheorem{cor}[thm]{{\bf Corollary}}
\newtheorem{prop}[thm]{{\bf Proposition}}

\numberwithin{equation}{section}

\newtheoremstyle{commondef}  
    {6pt}
    {6pt}
    {}
    {0em}
    {\bfseries}
    {}
    {.5em}
    {}
\theoremstyle{commondef}

\newtheorem*{remn}{{\bf Remark}}
\newtheorem{bem}[thm]{{\bf Remark}}

\renewenvironment{proof}{{\bf Proof}.}{\qed\\}

\makeatletter
\@namedef{subjclassname@2020}{\textup{2020} Mathematics Subject Classification}
\makeatother

\begin{document}

\begin{abstract}
We study a free boundary problem describing a tubular soap film bridge driven by an electrostatic force. To realize the electrostatic force, the soap film bridge is placed inside an outer metal cylinder and a voltage difference between this cylinder and the soap film is applied. Our model consists of a quasilinear parabolic
equation for the evolution of the film coupled with an elliptic equation for the electrostatic potential in the unknown domain
between outer cylinder and soap film bridge. For the rotationally symmetric case, we prove local well-posedness
of this free boundary problem by recasting it as a single quasilinear parabolic equation with a non-local
source term. Moreover, for large applied voltages, we show that solutions do not exist globally for a wide class of initial values. 
\end{abstract}

\maketitle

\allowdisplaybreaks

\section{Introduction}

The interplay between surface tension and electrostatics is the underlying mechanism of many processes taking place on small length-scales. There are various examples such as the fabrication of microstructures \cite{KW95}, electrowetting \cite{BB12} as well as the whole field of micro-electro-mechanical systems (MEMS) \cite{ FMCCS05, Pelesko03}. The latter includes tiny sensors and switches, used everywhere in modern technology. Probably also due to this great potential for applications, a variety of mathematical models for the interplay between surface tension and electrostatics have been proposed and investigated. These models comprise variational models for charged drops \cite{MN16,MNR22} as well as different types of MEMS-models, like singular equations \cite{BGP00,EGG10,Pelesko03}, free boundary problems \cite{ELW14, ELW15, LW17} or transmission problems \cite{LW23}. While many models put a special focus on the effect of electrostatics as the main destabilizing effect, our paper is concerned with a new mathematical model for a prototypical set-up in which also surface tension has the ability to break it. The set-up itself was suggested in \cite{Moulton08}, see also \cite{MP08, MP09}, and consists of a tiny soap film spanned between two parallel metal rings of equal size and subjected to an external electrostatic force. For the precise set-up we refer to Figure \ref{setup}. In contrast to the proposed model in \cite{Moulton08,MP09}, which describes static film deflections and consists of a singular ordinary differential equation, we focus on the dynamical behaviour of the soap film bridge which we describe rather by a free boundary problem: 

 \begin{figure}[h]
\begin{tikzpicture}
\begin{axis}[samples=200, axis x line=none, axis y line=none, domain=-1:1, clip=false]
\addplot[color=cyan, thick]({0.7*x},{-1+0.17*(1-x^2)^0.5});
\addplot[color=cyan, thick]({0.7*x},{1+0.17*(1-x^2)^0.5});
\addplot[color=red, thick]({1.1*x},{1+0.3*(1-x^2)^0.5});
\addplot[color=red!40, thick]({1.1*x},{-1+0.3*(1-x^2)^0.5});
\addplot[->, domain=-1.2:1.5]({0},{x});
\addplot[domain=-0.02:0.02]({x},{-1});
\addplot[domain=-0.02:0.02]({x},{1});

\addplot[color=red, thick]({-1.1},{x});
\addplot[color=red, thick]({1.1},{x});
\addplot[->, domain=-1.2:1.5]({0},{x});

\addplot[color=blue, thick]({cosh(x)/cosh(1)-0.3},{x});
\addplot[color=blue, thick]({-cosh(x)/cosh(1)+0.3},{x});

\addplot[color=cyan,thick]({0.7*x},{(1-0.17*(1-x^2)^0.5)});
\addplot[color=cyan, thick]({0.7*x},{-1-0.17*(1-x^2)^0.5});

\addplot[color=red,thick]({1.1*x},{(1-0.3*(1-x^2)^0.5)});
\addplot[color=red, thick]({1.1*x},{-1-0.3*(1-x^2)^0.5});

\end{axis}

\draw [->, bend angle=45, bend left]  (6.8,4) to (6.3,3.6);
\node[text width=3.5cm] at (8.3,4.7) {\begin{small}rigid cylinder \\ held at positive potential \end{small}};

\draw [->, bend angle=30, bend right]  (6.6,2.7) to (4.8,2);
\node[text width=3.5cm] at (8.3,2.7) {\begin{small} soap film \\ \ held at potential $0$ \end{small}};

\draw [->, bend angle=30, bend left]  (2.8,5.1) to (3,4.7);
\node[text width=3.5cm] at (2.6,5.3) {\begin{small} fixed boundary \end{small}};
\end{tikzpicture}

\caption[Soap Film Bridge in an Electric Field]{Depiction of the soap film (blue) inside an outer metal cylinder (red). The film, which is surrounded by air, is fixed at two parallel metal rings of equal size (light blue) whereas the remaining part of the film is free to move (dark blue). 
Applying a voltage between the film and the outer cylinder changes the shape of the film. \label{setup}
}
\end{figure}
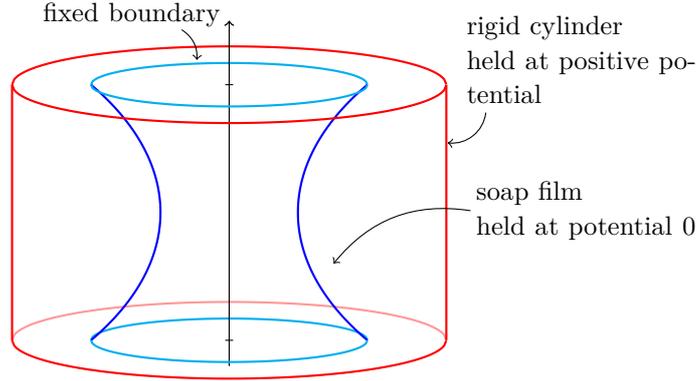

\subsection{The Mathematical Model}

We work in cylindrical coordinates and assume that the set-up is rotationally symmetric in which case the dimensionless problem can be described by two unknowns. The first one $u=u(t,z): [0,T) \times (-1,1) \rightarrow (-1,1)$ with $u(t,\pm1)=0$ and $T > 0$, models the soap film as a surface of revolution with profile $u+1$ for each fixed time. The second unknown is the electrostatic potential $\psi=\psi(z,r): \overline{\Omega(u)} \rightarrow \mathbb{R}$ , where 
$$\Omega(u) = \big \lbrace (z,r) \in (-1,1) \times (0,2) \, \vert \, u(z)+1 < r < 2 \, \big \rbrace$$ 
denotes the a-priori unknown space between soap film and outer cylinder, see Figure \ref{Cross Section} for a cross section of the set-up. We assume that the time evolution of $\psi$ is quasistatic (corresponding to the fact that time will only occur as a parameter in the equation \eqref{psigleichung42} for $\psi$ below) for which reason we always suppress the $t$-dependency of $\psi$ and $\Omega(u)$. Under the assumption that the problem is entirely driven by surface tension and electrostatics, we describe its dynamics by the following free boundary problem: The time-dependent film deflection $u$ solves the parabolic equation
 \begin{align}
 \begin{cases}
    \partial_t u -\sigma\, \partial_z \mathrm{arctan}  ( \sigma\partial_z u )&=  -\displaystyle\frac{1}{u+1} +\lambda\,(1+\sigma^2(\partial_z u)^2)^{3/2} \vert \partial_r \psi(z,u+1) \vert^2\,, \\
  \qquad \qquad \qquad u(t, \pm 1 )&=0 \,, \qquad -1 < u <1\,, \\
  \qquad \qquad \qquad u(0,z)&= u_0 \,, \qquad z \in (-1,1)\,,
 \end{cases} \label{filmdimensionless2}
\end{align}
with initial value $u_0$ satisfying $-1<u_0 <1$. Moreover, the electrostatic potential $\psi$ is given by 
\begin{align}
\begin{cases}\displaystyle
\frac{1}{r} \partial_r \left ( r \partial_r \psi \right ) + \sigma^2 \partial^2_z \psi&= 0  \quad \ \text{in} \quad \Omega(u)\,,  \\
 \ \, \qquad \qquad \qquad \qquad \psi &= h_u \quad \text{on} \quad \partial \Omega(u) \,,
\end{cases} \label{psigleichung42}
\end{align}
with suppressed $t$-dependency and a function $h_u$ which is constant $0$ on the film and $1$ on the outer cylinder. Neglecting the fringing field, as it is also done in free boundary problems for MEMS \cite{ELW14, ELW15, LW17}, we prescribe 
\begin{align}
h_u(z,r) = \frac{\ln \Big (\displaystyle\frac{r}{u(z)+1} \Big )}{\ln \Big (\displaystyle\frac{2}{u(z)+1} \Big )} \label{RBPsi42}
\end{align}
in the following. To come up with this boundary condition, we imagine that the film forms a cylinder and that the set-up is extended at the top and bottom to infinity. In this case, the electrostatic potential is explicitly known and gives \eqref{RBPsi42} if restricted back to the original set-up.
Finally, the important parameters of the model are $\sigma$ which gives the ratio of radii of the rings divided by their distance and $\lambda \in [0,\infty)$ which measures the strength of the applied voltage. If $\lambda=0$, no voltage is applied, and the soap film evolves according to mean curvature flow. For details on the (formal) derivation of \eqref{filmdimensionless2}-\eqref{RBPsi42} based on an energy consideration we refer to \cite{LSS24} and the references therein.

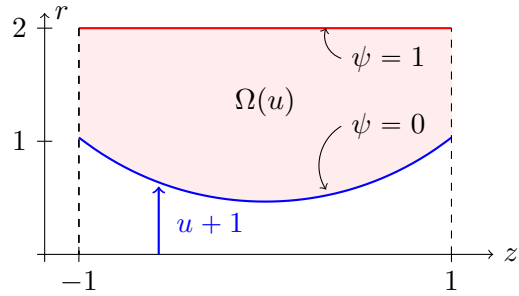
\begin{figure}[h]
\begin{tikzpicture}
\fill[white,rounded corners] (-3.5,-0.8) rectangle (3.3,3.5);
\fill[red!7] (-2.45,1.54) rectangle (2.45,3);
\draw[domain=-2.45:2.45, smooth, variable=\z, red!7, fill] plot ({\z},{cosh(0.5*\z)-0.3});
\draw[domain=-2.45:2.45, smooth, variable=\z, blue, thick] plot ({\z},{cosh(0.5*\z)-0.3});
\draw[red, thick] (-2.45,3) -- (2.45,3) ;
\draw[dashed](2.45,0) -- (2.45,3) ;
\draw[dashed] (-2.45,3) -- (-2.45,0) ;
\draw[dashed] (-2.45,3) -- (-2.45,0) ;
\draw[dashed] (-2.45,3) -- (-2.45,0) ;
\draw[->] (-3,0) -- (3,0) node[right] {$z$};
\draw (2.45,0.1) -- (2.45,-0.1) node[below] {$1$};
\draw (-2.45,0.1) -- (-2.45,-0.1) node[below] {$-1$};
\draw[->] (-2.9,-0.1) -- (-2.9,3.2) node[right] {$r$};
\draw (-2.8,1.5) -- (-3,1.5) node[left] {$1$};
\draw (-2.8,3) -- (-3,3) node[left] {$2$};
\node[above] at (0,1.7) {$\Omega(u)$};

\draw[->, blue, thick] (-1.4,0) -- (-1.4,0.9);
\node[blue, right] at (-1.3,0.4) {$u+1$};
\draw [<-, bend angle=45, bend right]  (0.8,3) to (1,2.6) node[right]{$\psi=1$};
\draw [<-, bend angle=45, bend left]  (0.8,0.86) to (1,1.7) node[right]{$\psi=0$};

\end{tikzpicture}
\caption[Cross Section Soap Film Bridge in an Electric Field]{Cross section of the soap film bridge in an electric field. Compared to Figure \ref{setup}, this picture is rotated by 90 degrees.}\label{Cross Section}
\end{figure}\color{black}

Our model shares similarities with the free boundary problem for MEMS in \cite{ELW15} from which we will also adapt some proofs to our setting. The main difference is that our source term \eqref{filmdimensionless2} is more complicated to control as it has no fixed sign and becomes singular in two cases $u= \pm 1$ instead of one. The case $u=-1$ corresponds to self-touching (more precisely: pinch-off) of the film and the case $u=1$ means that the film touches the outer cylinder. Moreover, we mention that the model from \cite{Moulton08,MP09} which consists of a singular ordinary differential equation may be derived from a stationary and not yet dimensionless version of \eqref{filmdimensionless2}-\eqref{RBPsi42} by assuming a small aspect ratio of the set-up. This means that the radii difference between the rings and the cylinder divided by the distance of the rings is set equal to zero.

\subsection{Main Results and Strategy}\label{Main Results}

We describe our results on existence and non-existence of local and global solutions. The first result reads as follows: 
\vspace{2mm}
\begin{thm}\label{localexistence}{\bf(Local Well-Posedness)}\\ Let $q \in (2,\infty)$, $\lambda \geq0$, $\sigma >0$ and $u_0 \in W^2_{q,D}(-1,1)$ with $1 > u_0(z) > -1$ for $z \in (-1,1)$. Then, there exists a unique maximal solution $(u ,\psi_u)$ to the coupled free boundary problem \eqref{filmdimensionless2}-\eqref{RBPsi42}
on the maximal interval of existence $\big [0, T_{max}\big )$ where $T_{max}=T_{max}(u_0)$ in the sense that 
\begin{align*}
 u \in C^1 \big ([ 0 ,T_{max}), L_q(-1,1) \big ) \cap C\big ( [0, T_{max}) , W^2_{q,D}(-1,1) \big ) 
\end{align*}
solves \eqref{filmdimensionless2} and $\psi_{u(t)} \in W^2_2 \big (\Omega(u(t))\big)$ solves \eqref{psigleichung42}-\eqref{RBPsi42} for each $t \in [0,T_{max})$.\\
\end{thm}

Here, functions in $W^2_{q,D}(-1,1)$ satisfy a Dirichlet boundary condition and we use the notation $\psi_u$ instead of $\psi$, which reason is explained below. To prove Theorem \ref{localexistence} we adapt \cite{ELW15} with smaller changes according to \cite{LW14b}. Our proof is based on semigroup theory (more precisely: its time-dependent counterpart \cite{AmannLQPP}) and Banach's fixed point theorem, and strongly relies on the quasilinear theory \cite{Amann93, AmannLQPP}. To apply this theory, we want to recast the free boundary problem \eqref{filmdimensionless2}-\eqref{RBPsi42} as a single parabolic equation for the film deflection $u$ only,
 \begin{align*}
 \begin{cases}
    \partial_t u -\sigma\, \partial_z \mathrm{arctan}  ( \sigma\partial_z u )&=  -\displaystyle\frac{1}{u+1} +\lambda\,g(u)\,, \\
  \qquad \qquad \qquad u(t, \pm 1 )&=0 \,, \qquad -1 < u <1\,, \\
  \qquad \qquad \qquad u(0,z)&= u_0 \,, \qquad z \in (-1,1)\,,
 \end{cases} 
\end{align*}
with $-1<u_0 <1$ and electrostatic force 
\begin{align}
  g(u):= (1+\sigma^2(\partial_z u)^2)^{3/2}\, \big \vert \partial_r \psi_u (z,u+1) \big \vert^2\,. \label{filmdimensionless0} 
\end{align}
This single parabolic equation is exactly equivalent to equation \eqref{filmdimensionless2}. The difference is now that we view the electrostatic force as a non-local map $\left [ u \mapsto g(u)\right ]$ between suitable function spaces. More precisely, for fixed time, this map should first take the function $u$ to the electrostatic potential $\psi_u$, which solves the elliptic equation \eqref{psigleichung42}-\eqref{RBPsi42} on the $u$-dependent domain $\Omega(u)$, and then do further manipulations with $\psi_u$ resulting in \eqref{filmdimensionless0}.  Note that the electrostatic potential occurs no longer as an equal unknown which justifies the notation $\psi_u$. Of course, it is not clear on which function spaces the map $[u \mapsto g(u)]$ is meaningful, and indeed our main work to prove Theorem \ref{localexistence} is to establish Lipschitz continuity of $g$ between fractional Sobolev spaces. To this end, we rely on elliptic regularity for non-smooth convex domains from \cite{LU68} for which we include a detailed proof in the Appendix \ref{ERCD}. With Lipschitz continuity of $g$ at hand, Theorem \ref{localexistence} follows from a refinement of a fixed point argument based on \cite{Amann93,AmannLQPP}.

\begin{remn}
We point out that the regularity of $[u \mapsto g(u)]$, which we provide in Section \ref{ES}, will be slightly different to the one usually required for local well-posedness results in the quasilinear theory, see for example \cite[Theorem 12.1]{Amann93} or also \cite[Theorem 1.1]{MW20}. This is due to the corners of $\Omega(u)$ and it is the reason for performing the whole fixed point argument by hand.
\end{remn}

To state our second result precisely, we require the set 
\begin{align}
S(\kappa):= \Big \lbrace v \in W^2_{q,D}(-1,1) \, \Big \vert \, \Vert v \Vert_{W^2_q(-1,1)} \leq 1/\kappa\,, \ -1+\kappa \leq v(z) \leq 1-\kappa \Big \rbrace \,, \label{Skappa}
\end{align}
for $\kappa >0$ and fixed $q >2$, which consists of film deflections which neither pinch-off nor touch the outer metal cylinder.\\

\begin{cor}\label{globex}{\bf (Global Existence Criterion)}\\
If for each $\tau > 0$, there exists $\kappa(\tau) \in (0,1)$ such that the unique maximal solution $u$ from Theorem \ref{localexistence} satisfies $u(t) \in S(\kappa(\tau))$ for all $t \in [0,T_{max}) \cap [0,\tau]$, then $u$ exists globally, that is $T_{max}=\infty$.\\
\end{cor}

Corollary \ref{globex} implies that if $T_{max} < \infty$, then the soap film pinches-off, touches the outer metal cylinder or the $\Vert \cdot \Vert_{W^2_{q,D}(-1,1)}$-norm of $u$ blows up. While the first two cases possess a direct physical interpretation, the norm blow-up corresponds more likely to a limitation of the model.

Finally, as our third contribution, we study the dynamical behaviour of the film deflection $u$ for large applied voltages $\lambda$. In this case, a dominance of the electrostatic force is expected resulting in non-existence of global solutions. Indeed, we can prove this non-existence under the additional assumption that $\sigma$ is not too small which a-priori excludes the possibility of pinch-off for many initial film deflections as we shall see below.
More precisely, we need to know that the stationary version of \eqref{filmdimensionless2} with $\lambda =0$,
\begin{align}
\begin{cases}
  &-\sigma \partial_z \mathrm{arctan} (\sigma \partial_z u) = - \displaystyle\frac{1}{u+1}\,,\\
  &u(\pm1)=0 \,, \quad -1 <u < 1\,,
\end{cases} \label{statlambda10}
\end{align}
 possesses at least one solution. Since \eqref{statlambda10} is the well-known minimal surface equation for a surface of revolution, it can be found in many textbooks, see for example \cite[p.\,282]{JLJ98}, that there exists a threshold value $\sigma_{crit} \approx 1.2$ such that \eqref{statlambda10} has at least one solution for $\sigma \geq \sigma_{crit}$ and no solution for $\sigma < \sigma_{crit}$.

Moreover, all stationary solutions have the shape of a (translated) catenoid, i.e. 
\begin{align}
u_{cat}(z) := \frac{\mathrm{cosh}(cz)}{\mathrm{cosh}(c)\ } - 1 \,, \qquad z \in (-1,1)\,, \label{catenoid}
\end{align} 
for $c>0$ with $\sigma = \frac{\mathrm{cosh}(c)}{c}$, and we may always assume, for fixed $\sigma \geq \sigma_{crit}$, that we have picked the smallest of such catenoids. \\

\begin{thm}\label{MainResultGlobalNonexistence}{\bf (Non-Existence of Global Solutions)}\\
Let $\sigma \geq \sigma_{crit}$. There exists $\lambda_{crit}(\sigma)>0$ such that for each $\lambda > \lambda_{crit}$ and each initial condition $u_0 \geq u_{cat}$, the corresponding solution $(u,\psi_u)$ to
\eqref{filmdimensionless2}-\eqref{RBPsi42} has a finite maximal time of existence $T_{max}(u_0) < \infty$. \\
\end{thm}

To outline the proof of Theorem \ref{MainResultGlobalNonexistence}, we fix $\sigma \geq \sigma_{crit}$ and a solution $(u,\psi_u)$ to \eqref{filmdimensionless2}-\eqref{RBPsi42} with initial value $u_0 \geq u_{cat}$. Then, the parabolic comparison principle guarantees that $u_0 \geq u_{cat}$ implies $u(t) \geq u_{cat}$, so that pinch-off of the soap film bridge is excluded. This makes it possible to consider the functional
 \begin{align*}
\mathcal{E}(t) := -\int_{-1}^1 \ln \big (u(t,z)+1 \big ) \, \mathrm{d} z\,, \qquad t \in [0,T_{max})\,, \qquad T_{max}=T_{max}(u_0)\,,
\end{align*}
which is bounded from below  
\begin{align*}
\mathcal{E}(t) \geq -2 \ln (2)\,, \qquad t \in [0,T_{max}) \,,
\end{align*}
while we aim at showing
\begin{align}
 \frac{\mathrm{d}}{\mathrm{d} t} \mathcal{E}(t) \leq - C < 0 \,, \qquad t \in [0,T_{max})\,, \label{LjapunovEst}
\end{align}
for $\lambda$ above a critical threshold value. Obviously, this is only possible if $T_{max} < \infty$. Thus, the main work is to establish \eqref{LjapunovEst} for which we will rely on several auxiliary estimates.\\
\color{black}

The question of non-existence of global solutions
for large $\lambda$ in variants of MEMS models has been previously studied in \cite{ELW13,ELW14,LW18b,L16}. In particular, a related result is contained in \cite{ELW13}, in which non-existence of global solutions to a MEMS model for large $\lambda$ is shown by deriving a more involved inequality for the functional $\tilde{\mathcal{E}}(t):= \int_{-1}^1 u(t,z) \,\mathrm{d} z$. 
Recall that in our case the right-hand side of \eqref{filmdimensionless2} contains two terms of opposite signs, which is the reason for working with the different energy functional $\mathcal{E}$. The term $-1/(u+1)$ will be controlled by the restriction $u_0 \geq u_{cat}$, while the positivity of
the electrostatic force $+ \lambda g(u)$ is accounted for by using the logarithm in the definition of $\mathcal{E}$.

\subsection{Outline of the Paper}\label{Outline}
In Section \ref{Notations}, we collect some notations and preliminaries. Then, in Section \ref{ES}, we prepare the proof of the first two main results by establishing Lipschitz continuity of $[u \mapsto g(u)]$ between fractional Sobolev spaces. To this end, we transform the unknown domain $\Omega(u)$ to a fixed rectangle. In Section~\ref{21123}, having Lipschitz continuity of $[u \mapsto g(u)]$ at hand, we prove Theorem \ref{localexistence} and Corollary \ref{globex}. Here, we rely on the reinterpretation of \eqref{filmdimensionless2}-\eqref{RBPsi42} in terms of $u$, semigroup theory and Banach's fixed point theorem. Finally, in Section \ref{section: Non-Existence} we prove Theorem \ref{MainResultGlobalNonexistence} on non-existence of global solutions for large voltages. The main text is supplemented by Appendix \ref{ERCD} where we give a detailed and new proof of an elliptic regularity result on general convex domains from \cite{LU68}, which forms the basis for the investigation of the map $[u \mapsto g(u)]$ in Section~\ref{21123}. \\

\section{Notations and Preliminaries}\label{Notations}

Let $U \subset \mathbb{R}^n$ be open and bounded with a Lipschitz boundary. For $p \in (1,\infty)$ and $s\in (0,2]$ with $s \neq 1/p$, we set\\
\begin{align*}
 W^s_{p,D}(U):= \begin{cases} \ \ \ W^s_p(U) \quad \qquad \qquad \qquad \qquad \ \ \,\text{for} \quad s\in (0,1/p)\,, \\
  \ \ \big \lbrace f \in W^s_p(U) \, \big \vert \, f=0 \ \text{on} \ \partial U \, \big \rbrace \quad\text{for} \quad s\in (1/p,2]\,.
\end{cases} 
\end{align*}
Moreover, we let $W^{-s}_{p^\prime,D}(U)$ be the dual space of $W^s_{p,D}(U)$ where $p^\prime$ denotes the dual exponent of $p$. We also require the following multiplication theorem for fractional Sobolev spaces: \\

\begin{thm}\label{multthm}
 Let $U \subset \mathbb{R}^n$ be a bounded domain with Lipschitz boundary. Let $m \in \mathbb{N}$ with $m \geq 2$ and $p,p_j \in (1,\infty)$ as well as $s,s_j \in (0,\infty)$ for $1 \leq j \leq m$. If $s \leq \mathrm{min} \lbrace s_j \rbrace$ and 
 \begin{align*}
 s- \frac{n}{p} \, < \begin{cases}
 &\displaystyle\sum_{s_j < n/p_j} \left (  s_j - \frac{n}{p_j} \right ) \qquad if \quad \underset{1 \leq j \leq m}{\mathrm{min}} \left \lbrace s_j - \frac{n}{p_j} \right \rbrace < 0 \,, \\
 &\displaystyle\underset{1 \leq j \leq m}{\mathrm{min}} \left \lbrace s_j - \frac{n}{p_j} \right \rbrace \qquad otherwise\,,
 \end{cases}
 \end{align*}
 then pointwise multiplication 
 $$\prod_{j=1}^m W^{s_j}_{p_j}(U) \rightarrow W^s_p(U)$$
 is continuous.
 \end{thm}
 
 \begin{proof} 
  This follows from \cite[Theorem 4.1, Remark 4.2\,(d)]{Amann91}, see also \cite[Theorem 7.1]{ELW15}.
 \end{proof}
 
We say that an operator $B: W^2_{q,D}(-1,1) \rightarrow L_q(-1,1)$ belongs to
$$\mathcal{H}  \big(W^2_{q,D}(-1,1), L_q(-1,1) \big )$$
if $-B$ generates an analytic semigroup on $L_q(-1,1)$ with domain $W^2_{q,D}(-1,1)$. Moreover, we require a more quantitative characterisation of generators of analytic semigroups from \cite{AmannLQPP}: For $\omega > 0$ and $k \geq 1$, an operator $B \in  \mathcal{H}\big(W^2_{q,D}(-1,1), L_q(-1,1) \big )$ belongs to the class
\begin{align} \mathcal{H}\big(W^2_{q,D}(-1,1), L_q(-1,1), k, \omega \big ) \label{QSE}
\end{align}
if $\omega + B$ is an isomorphism from $W^2_{q,D}(-1,1)$ onto $L_q(-1,1)$ and if 
\begin{align*}
\frac{1}{k} \leq \frac{\big \Vert(\mu +B ) v \big \Vert_{L_q(-1,1)} }{\vert \mu \vert \Vert v \Vert_{L_q(-1,1)} + \Vert v \Vert_{W^2_{q,D}(-1,1)}} \leq k \,, \quad \mathrm{Re} \mu \geq \omega \,, \quad v \in W^2_{q,D}(-1,1) \setminus \lbrace 0 \rbrace \,.
\end{align*}
The classes \eqref{QSE} make it possible to derive uniform estimates on semigroups, and hence to treat quasilinear parabolic equations with a non-local right-hand side.\\

Finally, if $E$ and $F$ are Banach spaces, we denote by $\mathcal{L}(E,F)$ the Banach space of bounded linear operators from $E$ to $F$,  by $\mathcal{L}_{is}(E,F)$ the set of isomorphisms from $E$ to $F$, and by $\mathcal{L}(E)$ the bounded linear operators from $E$ to $E$.

\section{Elliptic Subproblem}\label{ES}

In this section, we prove Lipschitz continuity of the map $[v \mapsto g(v)]$, where $g(v)$ denotes the electrostatic force from \eqref{filmdimensionless0}, and $v\in W^2_{q,D}(-1,1)$ with $q>2$ and $-1 < v(z) <1$ is a time-independent film deflection. The key step in the analysis of $[v \mapsto g(v)]$ is the investigation of $[v \mapsto \psi_v]$ with $\psi_v$ being the solution to the elliptic subproblem \eqref{psigleichung42}-\eqref{RBPsi42}.

As in \cite{LW17}, we note that this elliptic equation has a unique weak solution $\psi_v \in W^1_2\big (\Omega(v) \big)$ by Lax-Milgram Theorem, but that this regularity is not sufficient to define the electrostatic force $g(v)$ as it contains the square of the trace of the derivative of $\psi_v$. In addition, the Lax-Milgram Theorem provides no information on the dependency of $\psi_v$ on $v$.
To make the dependency of $\psi_v$ on $v$ accessible, we transform the domain $\Omega(v)$, on whose closure $\psi_v$ is defined, to an $v$-independent reference domain. More precisely, for a given film deflection $v \in W^2_{q,D}(-1,1)$ with $-1 > v(z) >1$ and $q >2$, we transform the domain $\Omega(v)$ to the fixed rectangle
 $$\Omega = (-1,1) \times \big ( 1,2 )$$
 via $T_v : \overline{\Omega(v)} \rightarrow \overline{\Omega}$ defined by 
 \begin{align}
  T_v (z,r) := \left ( z, \frac{r-2v(z)}{1-v(z)} \right )  \,, \qquad (z,r)\in \overline{\Omega(v)}\,. \label{deftrafoT}
 \end{align}
 
Due to the chain rule as well as transformation results for Sobolev functions \cite[Lemma 2.3.2]{Necas12}, we get that the electrostatic potential $\psi_v$ solves \eqref{psigleichung42}-\eqref{RBPsi42} weakly or strongly on $\Omega(v)$ if and only if $\phi_v:= \psi_v \circ (T_v)^{-1}$ is a weak or strong solution to 
 \begin{align}
 \begin{cases}
  L_v \phi_v &=0 \qquad \quad \text{in} \quad \Omega\,,\\
	\ \ \ \,	\phi_v &= \displaystyle\frac{\mathrm{ln}(r)}{\mathrm{ln}(2)} \quad \text{on} \quad \partial \Omega\,,
  \end{cases}  \label{ellsub}
 \end{align}
 where the transformed $v$-dependent differential operator $L_v$ is given by
 \begin{align}
  L_v w &:= \sigma^2 (1-v) \partial_z^2 w - 2 \sigma^2 \, \partial_z v \,  ( 2-r  ) \partial_r\partial_z w \nonumber \\
  &\quad\ \ +\frac{1+\sigma^2 (\partial_z v)^2  ( 2-r  )^2}{1-v}\, \partial_r^2 w \nonumber \\
  &\quad\ \  + \left [ - \sigma^2  ( 2 - r  ) \Big (  \partial_z^2 v +  \frac{2 (\partial_z v)^2 }{1-v} \Big ) + \frac{1}{2v +(1-v)r } \right ] \partial_r w\,. \label{Lvnondivergence}
 \end{align}
 In divergence form this operator reads
 \begin{align}
L_v w = \mathrm{div} \left (A(v) \nabla w \right ) + d(v)\cdot \nabla w   \label{Lvdivergence}
 \end{align}
 with
 \begin{align*}
  A(v) = [a_{ij}(v)]_{i,j=1}^2 &:= \begin{pmatrix}
          \sigma^2 (1-v)  & - \sigma^2  \,\partial_z v\, (2-r ) \\
         - \sigma^2 \,\partial_z v \, (2-r )  & \displaystyle\frac{1+\sigma^2 (\partial_z v)^2(2-r)^2}{1-v}
         \end{pmatrix}
\,, \\
d(v) = \begin{pmatrix} d_1(v) \\ d_2(v) \end{pmatrix}&:=\begin{pmatrix}  0\\ \displaystyle\frac{1}{2v+(1-v)r}\end{pmatrix} \,.
 \end{align*}
 
To prevent the operator $L_v$ from being degenerate or singular, we will study the dependency of \eqref{ellsub} on $v$ only on the sets $S(\kappa)$ defined in \eqref{Skappa}.

\subsection{Solution Theory}\label{21123b}
  
The aim of this subsection is threefold: We present the weak and strong solution theory for the problem
 \begin{align}
  \begin{cases}
   L_v \Phi_v &= F \quad \text{in} \quad \Omega\,, \\
   \ \ \ \Phi_v &= 0 \quad \text{on} \quad \partial \Omega
  \end{cases} \label{selliptic:eq1}
 \end{align}
with $F$ in $W^{-1}_{2,D}(\Omega)$ or $L_2(\Omega)$ respectively, which is closely related to the transformed problem \eqref{ellsub},
we derive a-priori estimates for $\Phi_v$ holding uniformly on $S(\kappa)$, and we use interpolation theory to improve these a-priori estimates. The applied methods are similar to those leading to \cite[Lemma 2.2]{ELW15}.\\

Before we start, let us note that the transformed operator $-L_v$ is again uniformly elliptic with ellipticity constant independent of $v \in S(\kappa)$:

\begin{lem}\label{uniformelliptic} There exists a constant $\alpha=\alpha(\kappa)>0$ such that
\begin{align*}
 \alpha \vert \xi \vert^2 \leq \xi^T A(v) \xi \leq \frac{1}{\alpha} \vert \xi \vert^2\,, \qquad \xi \in \mathbb{R}^2 \,, \quad (z,r) \in \Omega \,, \quad v \in S(\kappa)\,.
\end{align*}
\end{lem}

\begin{proof}
The real eigenvalues $\mu_{\pm}$ of $A(v)$ satisfy
\begin{align*}
 \mathrm{tr} \big (A(v) \big )= \mu_+ + \mu_- \,, \qquad \det \big (A(v) \big )= \mu_+ \mu_-.
\end{align*}

Computing the expressions for the trace and the determinant explicitly, we easily see that there exists $\alpha(\kappa) > 0$ with 
\begin{align*}
\frac{1}{\alpha(\kappa)}\geq \mathrm{tr} \big (A(v) \big)\geq \mu_+ \geq \mu_- \geq \frac{\det \big (A(v) \big )}{\mathrm{tr}\big (A(v) \big )} \geq \alpha(\kappa) > 0 
\end{align*}
for all $(z,r) \in \Omega$ and $v \in S(\kappa)$.
\end{proof}

{\bf Weak Solutions.}
We consider weak solutions to \eqref{selliptic:eq1} and corresponding a-priori estimates.
Though existence and uniqueness results for solutions are usually supplemented by a-priori estimates, see \cite[Corollary 8.7, Lemma 9.17]{GT98} and \cite[Theorem 6.2.6]{Evans10}, we have to repeat the arguments to include the $v$-dependency.
 
 \begin{lem}\label{selliptic1}\label{selliptic2}
 For each $v\in S(\kappa)$ and each $F \in W^{-1}_{2,D}(\Omega)$, there exists a unique weak solution $\Phi_v \in W^1_{2,D}(\Omega)$ to \eqref{selliptic:eq1}, i.e. to the equation 
  \begin{align*}
  \begin{cases}
   L_v \Phi_v &= F \quad \text{in} \quad \Omega\,, \\
   \ \ \ \Phi_v &= 0 \quad \text{on} \quad \partial \Omega\,. 
  \end{cases}
 \end{align*}
Moreover, there exists $C_1(\kappa)>0$ (independent of $F$, $\Phi_v$ and $v$) such that 
  \begin{align}
   \Vert \Phi_v  \Vert_{W^1_{2,D}(\Omega)}  \leq C_1(\kappa) \Vert F \Vert_{W^{-1}_{2,D}(\Omega)} \,. \label{selliptic:eq2}
  \end{align}
 \end{lem}

 \begin{proof}
The existence of a unique weak solution to problem \eqref{selliptic:eq1} is a consequence of \cite[Theorem 8.3]{GT98}. So it remains to prove \eqref{selliptic:eq2}:\\
{\bf(i)} As a first step, we show the existence of $C(\kappa) > 0$ with 
  \begin{align}
  \Vert \Phi_v \Vert_{W^1_2(\Omega)} &\leq C(\kappa)\, \big ( \Vert \Phi_v \Vert_{L_2(\Omega)} +\Vert F \Vert_{W_{2,D}^{-1}(\Omega)} \big ) \label{selliptic2.eq1}
 \end{align}
 for each $v \in S(\kappa)$ and $F \in W^{-1}_{2,D}(\Omega)$. To this end, we test the weak formulation of \eqref{selliptic:eq1} with $\Phi= \Phi_v$ resulting in
  \begin{align*}
   \int_\Omega \nabla \Phi^T A(v) \nabla \Phi\, \mathrm{d} (z,r) =\int_\Omega \big (d(v) \cdot \nabla\Phi \big ) \Phi \, \mathrm{d} (z,r) -\langle F , \Phi \rangle_{W^1_2(\Omega)}\,.
  \end{align*}
 Combining now the uniform ellipticity of $-L_v$ with $v$-independent ellipticity constant $\alpha(\kappa)>0$ from Lemma \ref{uniformelliptic} with Friedrich's inequality yields
 \begin{align*}
  \Vert \Phi \Vert_{W^1_2(\Omega)}^2 &\leq C(\kappa) \left (  \Big \vert \int_\Omega\big (d(v) \cdot \nabla\Phi \big ) \Phi \, \mathrm{d} (z,r) \Big \vert + \Vert F \Vert_{W_{2,D}^{-1}(\Omega)} \Vert \Phi \Vert_{W^1_2(\Omega)} \right )
 \end{align*}
 for some $C(\kappa) >0$. Finally, the fact that $\Vert d(v) \Vert_{\infty}$ is uniformly bounded on $S(\kappa)$ together with Hölder's inequality and Young's inequality gives
 \begin{align*}
  \Vert \Phi \Vert_{W^1_2(\Omega)}^2 &\leq C(\kappa) \big ( \Vert \Phi \Vert_{L_2(\Omega)}^2 +\Vert F \Vert_{W_{2,D}^{-1}(\Omega)} \Vert \Phi \Vert_{W^1_2(\Omega)} \big )
 \end{align*}
for some new $C(\kappa)>0$, which is obviously equivalent to \eqref{selliptic2.eq1}. \\
{\bf(ii)} Next, the $L_2$-norm of $\Phi_v$ on the right-hand side of \eqref{selliptic2.eq1} has to be eliminated. However, for this we can proceed by contradiction as in \cite[Lemma 9.17]{GT98}, and we refer to \cite[Lemma 3.2]{LSS24} for details.
 \end{proof}

{\bf Regularity Step: Strong Solutions.} We establish that $\Phi_v$ is a strong solution to \eqref{selliptic:eq1} if the right-hand side $F$ is more regular. Since $\Omega$ is a rectangle, i.e. a domain with corners, this result does not follow from standard elliptic regularity theory,
but from Theorem \ref{5}. 
 
\begin{lem}\label{selliptic3}
For each $v\in S(\kappa)$ and each $F \in L_2(\Omega)$, there exists a unique strong solution $\Phi_v \in W^2_{2,D}(\Omega)$ to \eqref{selliptic:eq1}, i.e to the equation 
 \begin{align*}
  \begin{cases}
   L_v \Phi_v &= F \quad \text{in} \quad \Omega\,, \\
   \ \ \ \Phi_v &= 0 \quad \text{on} \quad \partial \Omega\,.
  \end{cases}
 \end{align*}
 Moreover, there exists $C_2(\kappa)>0$ (independent of $F$, $\Phi_v$ and $v$) such that 
  \begin{align*}
   \Vert \Phi_v  \Vert_{W^2_{2}(\Omega)}  \leq C_2(\kappa) \Vert F \Vert_{L_2(\Omega)} \,.
  \end{align*}
 \end{lem}
 
 \begin{proof} By Lemma \ref{selliptic2}, we find $\Phi_v \in W^1_{2,D}(\Omega)$ being the unique weak solution to
 \begin{align*}
 \begin{cases}
  \mathrm{div}\big ( A(v)\nabla w \big ) &=  F - d(v) \cdot \nabla \Phi_v  \quad \text{in} \quad \Omega\,, \\
  \quad \quad \quad \quad \ \, w&= 0 \quad \text{on} \quad \Omega \,.
 \end{cases}
\end{align*}
From the facts that $W^1_{q}(-1,1)$ is a Banach algebra and $W^1_q(-1,1) \hookrightarrow C([-1,1])$, we deduce that the coefficients of $-L_v$ satisfy
\begin{align*}
\sum_{i,j=1}^2 \Vert a_{ij}(v) \Vert_{W^1_q(\Omega)} +\sum_{i=1}^2 \Vert d_i(v) \Vert_{\infty} \leq C(\kappa)\,, \qquad v \in S(\kappa)\,.
\end{align*}

 Because the ellipticity constant $\alpha(\kappa) >0$ of $-L_v$ is independent of $v \in S(\kappa)$, see Lemma \ref{uniformelliptic}, we deduce from Theorem \ref{5}
 that $\Phi_v$ belongs to $W^2_{2,D}(\Omega)$ and that there exists $C(\kappa) > 0$ with
  \begin{align*}
   \Vert \Phi_v \Vert_{W_2^2(\Omega)} \leq C(  \kappa) \big ( \Vert \Phi_v \Vert_{W^1_2(\Omega)} + \Vert F \Vert_{L_2(\Omega)} \big)\,.
  \end{align*}
 Since $\Vert F \Vert_{W^{-1}_{2,D}(\Omega)} \leq \Vert F \Vert_{L_2(\Omega)}$, it follows from Lemma \ref{selliptic2} that 
 \begin{align*}
  \Vert \Phi_v \Vert_{W_2^2(\Omega)} \leq C_2( \kappa) \Vert F \Vert_{L_2(\Omega)} \,
 \end{align*}
 for some $C_2(  \kappa) >0$, and the proof is complete.
 \end{proof}
 
In summary, the previous two Lemmata \ref{selliptic2} and \ref{selliptic3} ensure unique weak and strong solvability of \eqref{selliptic:eq1}. More precisely, for $v \in S(\kappa)$, the operator
 \begin{align}
  L_D(v) \Phi := L_v \Phi\,, \qquad \Phi \in W^1_{2,D}(\Omega) \label{defLDv}
 \end{align}
 \vspace{-2mm}satisfies
$$L_D(v) \in \mathcal{L}_{is} ( W^1_{2,D}(\Omega), W^{-1}_{2,D}(\Omega)) \cap \mathcal{L}_{is}(W^2_{2,D}(\Omega) , L_2(\Omega)),$$
and its inverse $L_D(v)^{-1}$ is uniformly bounded for $v \in S(\kappa)$. \\

From a solution to \eqref{selliptic:eq1} one easily obtains a solution to the transformed electrostatic problem:
Noting that $f_v:=L_v \frac{\ln(r)}{\ln(2)}$ belongs to $L_2(\Omega)$ one finds that 
\begin{align}
\phi_v := -L_D(v)^{-1} f_v +\frac{\ln(r)}{\ln(2)} \in W^2_2(\Omega)  \label{defphiv}
\end{align}
is the unique strong solution to the transformed electrostatic problem \eqref{ellsub}. Thanks to 
 \begin{align}
 \Vert f_v \Vert_{L_2(\Omega)} \leq C( \kappa) \,, \qquad v \in S(\kappa)\,,  \label{Estimatefv}
\end{align}
and the uniform estimates on $L_D(v)^{-1}$, the function $\phi_v$ also satisfies a uniform estimate
\begin{align}
\Vert \phi_v \Vert_{W^2_2(\Omega)} \leq C( \kappa) \,, \qquad v \in S(\kappa)\,.\label{Estimatephiv}\\ \nonumber
\end{align}

\begin{bem}\label{refl}
 We briefly comment on the regularity of the original electrostatic potential $\psi_v=\phi_v \circ T_v$ solving 
 \begin{align*}
  \begin{cases}
   \displaystyle\frac{1}{r} \partial_r (r \partial_r \psi_v)+ \partial_z^2 \psi_v&=0  \quad \text{in} \quad \Omega(v)\,,\\
   \qquad \qquad \qquad \psi_v&=h_v  \quad \text{on} \quad \partial \Omega(v)\,,
  \end{cases}
 \end{align*}
 where we set $\sigma =1$ in this remark, and $h_v$ is given by \eqref{RBPsi42}.
 Due to the corners of $\Omega(v)$, one might expect the regularity $\psi_v \in W^2_2(\Omega(v)) \cap C^\infty \big (\overline{\Omega(v)}\setminus \lbrace  (\pm 1, 1), (\pm 1,2) \rbrace \big) $ to be optimal in general. However, one can show that $\psi_v$ is smooth up to the 
 boundary in $(\pm 1,2)$. In addition, if $v \in W^3_{\infty}(-1,1)$ with $v(\pm1)=v_z(\pm1)=v_{zz}(\pm 1)=0$, then $\psi_v \in C^{2,\alpha}(\overline{\Omega(v)})$ for any $\alpha \in (0,1)$, i.e. $\psi_v$ is a classical solution. This follows from the Schwarz reflection principle \cite[Exercise 2.4]{GT98} and Schauder Theory, see \cite[Lemma 6.18]{GT98}.\\
\end{bem}

{\bf Fine Tuning Via Interpolation.} Finally, using interpolation theory, we get an improved norm estimate for the inverse of $L_D(v)$, which results in better estimates for $[v \mapsto \phi_v]$ in the next subsection.
The proof is exactly the same as in \cite{ELW15}.

\begin{prop}\label{IntA}
Given $\theta \in [0,1] \setminus \lbrace 1/2 \rbrace$, there is a constant $C_3( \kappa) >0$ such that
\begin{align*}
 \Vert L_D(v)^{-1} \Vert_{\mathcal{L}(W_{2,D}^{\theta-1}(\Omega) , W_{2,D}^{\theta+1} (\Omega))} \leq C_3( \kappa)\,, \qquad v \in 
S(\kappa)\,.
\end{align*}
\end{prop}

\begin{proof}
See \cite[Lemma 2.3]{ELW15}.
\end{proof}

\color{black}

 \subsection{Regularity of the Electrostatic Force}\label{21123c}
 
In this subsection, we prove Lipschitz continuity and analyticity of the electrostatic force $[v \mapsto g(v)]$ for which we adapt \cite{ELW15}.\\

For convenience, we recall the notation 
\begin{align*}
S(\kappa)= \Big \lbrace v \in W^2_{q,D}(-1,1) \, \Big \vert \, \Vert v \Vert_{W^2_q(-1,1)} \leq 1/\kappa\,, \ -1+\kappa \leq v(z) \leq 1-\kappa \Big \rbrace 
\end{align*}
for $\kappa >0$ and $q >2$ while $q^\prime$ denotes the dual exponent of $q$.\\

The desired Lipschitz continuity is proven in several steps. First, we derive continuity properties of $[v \mapsto L_v]$ where $L_v$ is defined in \eqref{Lvnondivergence}. Subsequently, we establish continuity of $[v \mapsto \phi_v]$, and finally, we transfer the continuity properties to $[v \mapsto g(v)]$. 
The regularity of $[v \mapsto L_v]$ follows as in \cite[Lemma 2.4]{ELW15}.

\begin{lem}\label{ALipschitz}
 Given $\xi \in [0, 1/q^\prime)$ and $\alpha \in (\xi,1)$, there exists $C_4( \kappa)$ such that
 \begin{align*}
  \Vert L_v - L_w \Vert_{\mathcal{L}(W^2_2(\Omega), W_{2,D}^{-\alpha}(\Omega))} \leq C_4( \kappa) \Vert v-w \Vert_{W_q^{2-\xi}(-1,1)} 
 \end{align*}
 for all $v,w \in S(\kappa)$.
\end{lem}

 \begin{proof}
Let $v,w \in S(\kappa)$ and $\Phi \in W^2_2(\Omega)$. Then,
$L_v \Phi  \in L_2(\Omega) \subset W_{2,D}^{-\alpha}(\Omega)$ where the critical term $- \sigma^2 (2-r) \,\partial_z^2 v \,\partial_r \Phi$ of $L_v \Phi$ belongs to $L_2(\Omega)$ thanks to Hölder's inequality and the embedding $ W_2^1(\Omega) \hookrightarrow L_{\frac{2q}{q-2}}(\Omega)$. 
For $\psi \in W^\alpha_{2,D}(\Omega)$, the definition of $L_v$ in non-divergence form yields 
  \begin{align*}
 &\int_\Omega \big [ (L_v-L_w)\Phi \big ] \psi \, \mathrm{d} (z,r) \\
   &= \sigma^2 \int_\Omega  [w-v] \,\partial_z^2 \Phi\, \psi \,\mathrm{d}(z,r)\\
   &\quad \ \ -2\sigma^2 \int_\Omega  ( 2 -r  ) \,[\partial_z v-\partial_z w ]\, \partial_z\partial_r \Phi \,\psi \,\mathrm{d} (z,r) \\
   &\quad \ \ + \int_\Omega \bigg ( \frac{1+\sigma^2 (\partial_z v)^2 (2 -r)^2 }{1-v}-  \frac{1+\sigma^2 (\partial_z w)^2 (2 -r)^2 }{1-w} \bigg)\, \partial_r^2 \Phi \,\psi \, \mathrm{d} (z,r)\\
   &\quad \ \ - \sigma^2 \int_\Omega  (2 -r ) \,[\partial_z^2v-\partial_z^2w]\, \partial_r \Phi \,\psi \, \mathrm{d} (z,r) \\
   &\quad \ \ - 2\sigma^2\int_\Omega ( 2 -r  ) \,\bigg ( \frac{(\partial_z v)^2}{1-v} - \frac{(\partial_z w)^2}{1-w} \bigg )\, \partial_r \Phi\, \psi 
\, \mathrm{d}(z,r) \\
   &\quad \ \ +\int_\Omega \bigg ( \frac{1}{2v+(1-v)r} - \frac{1}{2w+(1-w)r} \bigg )\, \partial_r \Phi \,\psi \,\mathrm{d}(z,r) \\
   &=:I + II+\dots + VI \,.
  \end{align*}
We point out that $IV$ is the critical term since it contains the already mentioned second weak derivatives of $v$ (and $w$). Therefore, we will only treat $IV$ in detail. The other five terms can be estimated by using Hölder's inequality, the embedding $W^{2-\xi}_q(-1,1) \hookrightarrow C^1([-1,1])$ as well as the definition of $S(\kappa)$, which results in
\begin{align*}
\vert I \vert + \vert II \vert + \vert III \vert &+ \vert V \vert + \vert VI \vert \\
&\leq C \Vert v-w\Vert_{W_q^{2-\xi}(-1,1)}  \Vert \Phi \Vert_{W^2_2(\Omega)} \Vert \psi \Vert_{W^\alpha_{2,D}(\Omega)}\,.
\end{align*}

For $IV$, a simple application of Hölder's inequality combined with Sobolev's embedding theorem only yields existence of the integral but not the desired estimate. Instead, we argue as follows: Due to Fubini
's Theorem and the fact that $\partial_z^2 \in \mathcal{\mathcal{L}}\big ( W^{2-\xi}_{q}(-1,1), W^{-\xi}_{q,D}(-1,1) \big)$ by \cite[Theorem 1.4.4.6]{Grisvard85} (as $1-\xi \neq 1/q$), we find
\begin{align*}
\vert IV \vert &\leq \sigma^2 \bigg \vert \int_\Omega  ( 2-r)\,[\partial_z^2v- \partial_z^2 w ] \,\partial_r \Phi\, \psi \, \mathrm{d}(z,r) \ \bigg \vert  \\
&=\sigma^2  \bigg \vert \int_{-1}^1 [\partial_z^2v- \partial_z^2 w ] (z) \bigg ( \int_1^2  ( 2 -r   ) \partial_r  \Phi(z,r)\psi(z,r) \,\mathrm{d} r \bigg ) \,\mathrm{d} z \bigg \vert  \\
&\leq \sigma^2 \Vert \partial_z^2v- \partial_z^2 w \Vert_{W^{-\xi}_{q,D}(-1,1)} \bigg \Vert \int_1^2  ( 2 -r  ) \partial_r \Phi(\, \cdot \, ,r)\psi(\, \cdot \,,r) \, \mathrm{d} r \bigg \Vert_{W^\xi_{q^\prime}(-1,1)}  \\
&\leq C\,\sigma^2  \Vert v -w \Vert_{W^{2-\xi}_q(-1,1)} \bigg \Vert \int_1^2  (2-r  ) \partial_r \Phi(\, \cdot \, ,r)\psi(\, \cdot \,,r) \, \mathrm{d} r \bigg \Vert_{W^\xi_{q^\prime}(-1,1)}\,. 
\end{align*}
Here, we also used the fact that $W_{q^\prime,D}^\xi(-1,1)=W^\xi_{q^\prime}(-1,1)$ due to the choice $\xi < 1/q^\prime$ so that the dual space of $W^\xi_{q^\prime}(-1,1)$ coincides with $W^{-\xi}_{q,D}(-1,1)$.
Next, we apply the real interpolation method to obtain
\begin{align*}
\bigg \Vert \int_1^2  (2-r  ) \partial_r \Phi(\, \cdot \, ,r)\psi(\, \cdot \,,r) \, \mathrm{d} r \bigg \Vert_{W^\xi_{q^\prime}(-1,1)} \leq C \big \Vert (2-r) \partial_r \Phi \, \psi \big \Vert_{W^\xi_{q^\prime}(\Omega)} \,,
\end{align*} 
see \cite[Lemma A.2]{LSS24} for details, from which we deduce further that
 \begin{align*}
 \vert IV \vert \leq C\,\sigma^2  \Vert v-w \Vert_{W_q^{2-\xi}(-1,1)} \big \Vert (2-r) \partial_r \Phi \, \psi \big \Vert_{W^\xi_{q^\prime}(\Omega)} \,.
 \end{align*}
Finally, the Multiplication Theorem \ref{multthm} ensures
\begin{align*}
 W^1_2(\Omega) \cdot W^1_2(\Omega) \cdot W^\alpha_2 (\Omega) \hookrightarrow W^\xi_{q^\prime}(\Omega)\,,
\end{align*}
and we arrive at
\begin{align*}
\vert IV \vert &\leq C \Vert v-w \Vert_{W_q^{2-\xi}(-1,1)}  \Vert 2-r  \Vert_{W^1_2(\Omega)} \Vert \partial_r \Phi \Vert_{W^1_2(\Omega)} \Vert \psi \Vert_{W^\alpha_{2,D}(\Omega)} \\
&\leq C \Vert v-w\Vert_{W_q^{2-\xi}(-1,1)}  \Vert \Phi \Vert_{W^2_2(\Omega)} \Vert \psi \Vert_{W^\alpha_{2,D}(\Omega)}\,.\\
\end{align*}
Summing up the estimates for $I$ to $VI$, we have shown that
\begin{align*}
\Big \vert \int_\Omega \big [(L_v - L_w)\Phi \big ] \,\psi \, \mathrm{d} (z,r) \Big \vert \leq C_4( \kappa)\Vert v-w \Vert_{W_q^{2-\xi}(-1,1)} \, \Vert \Phi \Vert_{W^2_2(\Omega)} \, \Vert \psi \Vert_{W^\alpha_{2,D}(\Omega)} \,.
\end{align*}
Taking the supremum over $\psi \in W^\alpha _{2,D}(\Omega)$ with $\Vert \psi \Vert_{W_{2,D}^\alpha(\Omega)} \leq 1$, we get 
\begin{align*}
\big \Vert (L_v-L_w) \Phi \big \Vert_{W^{-\alpha}_{2,D}(\Omega)} \leq C_4( \kappa)\Vert v-w \Vert_{W_q^{2-\xi}(-1,1)} \Vert \Phi \Vert_{W^2_2(\Omega)} \,,
\end{align*}
and thus
\begin{align*}
\Vert L_v-L_w\Vert_{\mathcal{L}(W^2_2(\Omega) , W^{-\alpha}_{2,D}(\Omega))} \leq C_4( \kappa) \Vert v-w \Vert_{W_q^{2-\xi}(-1,1)}
\end{align*}
as claimed. 
\end{proof}

Next, we study the dependence of $\phi_v$ on $v$. The result is the analogue to \cite[Lemma 2.6]{ELW15}.

\begin{lem}\label{EstimateSolution}
Let $\xi \in [0,1/q^\prime)$ and $\alpha \in (\xi,1)$ with $\alpha \neq 1/2$ be given. Then, there exists $C_{5}( \kappa)$ such that 
\begin{align*}
\Vert \phi_v- \phi_w \Vert_{W_{2,D}^{2-\alpha}(\Omega)} \leq C_{5}( \kappa) \Vert v-w \Vert_{W_q^{2-\xi}(-1,1)}  \,, \qquad v,w \in S(\kappa)\,.
\end{align*}
\end{lem}
\begin{proof}
Let us recall from \eqref{defphiv} that 
\begin{align*}
\phi_v = -L_D(v)^{-1} f_v +\frac{\ln(r)}{\ln(2)} \,, \qquad f_v= L_v \frac{\ln(r)}{\ln(2)}\,.
\end{align*}
First, we deduce from Lemma \ref{ALipschitz} that
\begin{align}
\Vert f_v - f_w \Vert_{W_{2,D}^{-\alpha}(\Omega)} &\leq \Vert L_v-L_w \Vert_{\mathcal{L}(W^2_2(\Omega), W_{2,D}^{-\alpha}(\Omega))} \bigg \Vert \frac{\ln(r)}{\ln(2)} \bigg \Vert_{W_2^2(\Omega)} \nonumber\\
&\leq C( \kappa) \Vert v-w \Vert_{W_q^{2-\xi}(-1,1)} \,. \label{EstimateRHS}
\end{align}

Next, we write 
$$\phi_v-\phi_w =- L_D(v)^{-1} (f_v-f_w) + \big (L_D(w)^{-1} - L_D(v)^{-1}\big)f_w\,.$$ 

Then, a combination of \eqref{EstimateRHS} with Proposition \ref{IntA} (for $\theta=1-\alpha \neq 1/2$ and $\theta=1$) as well as Lemma \ref{ALipschitz} yields
\begin{align*}
\Vert \phi_v - \phi_w \Vert_{W_{2,D}^{2-\alpha}(\Omega)} &\leq \Vert L_D(v)^{-1}(f_v-f_w) \Vert_{W_{2,D}^{2-\alpha}(\Omega)} + \Vert (L_D(v)^{-1} -L_D(w)^{-1}) f_w \Vert_{W_{2,D}^{2-\alpha}(\Omega)} \\
&\leq \Vert L_D(v)^{-1} \Vert_{\mathcal{L}(W_{2,D}^{-\alpha}(\Omega), W^{2-\alpha}_{2,D} (\Omega))} \Vert f_v - f_w \Vert_{W^{-\alpha}_{2,D}(\Omega)} \\
&\quad \ \ + \Vert L_D(v)^{-1}(L_w-L_v)L_D(w)^{-1} f_w \Vert_{W^{2-\alpha}_{2,D}(\Omega)}\\
&\leq C( \kappa) \Vert v-w\Vert_{W_q^{2-\xi}(-1,1)}   + \Vert L_D(v)^{-1} \Vert_{\mathcal{L}(W^{-\alpha}_{2,D}, W^{2-\alpha}_{2,D}(\Omega))}\\
&\quad \ \ \times \Vert L_w-L_v \Vert_{\mathcal{L}(W_{2}^2(\Omega) , W_{2,D}^{-\alpha} (\Omega))} \Vert L_D(w)^{-1} \Vert_{\mathcal{L}(L_2(\Omega) , W^2_{2,D}(\Omega))} \Vert f_w \Vert_{L_2(\Omega)} \\
&\leq C( \kappa)\Vert v-w \Vert_{W_q^{2-\xi}(-1,1)} \big (1 + \Vert f_w \Vert_{L_2(\Omega)} \big ) \,.
\end{align*}
Finally, estimate \eqref{Estimatefv} ensures that the $L_2$-norm of $f_w$ is uniformly bounded on $S(\kappa)$, and the assertion follows.
 
\end{proof}

Having established continuity properties of $[v \mapsto \phi_v]$, we turn to the main issue of this section and provide Lipschitz continuity of the electrostatic force $[ v \mapsto g(v)]$. The result is an adaptation of \cite[Proposition 2.1]{ELW15}: 

\begin{prop}\label{LipschitzRHS}
Let $q \in (2,\infty)$, $\kappa \in (0,1)$ and $\lambda, \sigma > 0$. For $\xi \in [0,1/2)$ and $\nu \in [0, 1/2-\xi)$, the map
$$[v \mapsto g(v)]\,: \quad S(\kappa)\rightarrow W^\nu_{2,D}(-1,1)$$
is bounded, and there exists a constant $C_{6}( \kappa) >0$ such that
\begin{align}
\Vert g (v) &- g(w) \Vert_{W_{2,D}^\nu(-1,1)} \leq  \,C_6( \kappa) \Vert v-w \Vert_{W_{q,D}^{2-\xi}(-1,1)}  \label{eqEstimateRHS}
\end{align}
as well as 
\begin{align}
 \bigg \Vert \frac{1}{v+1} -\frac{1}{w+1} \bigg \Vert_{W_{2,D}^\nu(-1,1)} \leq  C_6( \kappa) \Vert v-w \Vert_{W_{q,D}^{2-\xi}(-1,1)} \,, \qquad v,w \in S(\kappa)\,. \label{eqEstimateRHSb}
\end{align}

\end{prop}
\begin{proof} {\bf(i)} As a first step, we express the electrostatic force
$$g(v)= \big (1+\sigma^2 (\partial_z v )^2 \big)^{3/2} \,\big \vert \partial_r \psi_v\big ( z,v+1 \big ) \big \vert^2$$
defined in \eqref{filmdimensionless0} in terms of the transformed electrostatic potential $\phi_v$. To this end, we recall from \eqref{deftrafoT} that
\begin{align*}
 \psi_v(z,r)=\phi_v\big (T_v(z,r)\big)= \phi_v \Big( z , \, \frac{r-2v(z)}{1-v(z)} \Big ) \,, \qquad (z,r) \in \Omega(v)\,,
\end{align*}
and consequently 
\begin{align*}
 \partial_r \psi_v \big (z,v(z)+1 \big) = \frac{\partial_r \phi_v(z,1)}{1-v(z)}\,, \qquad z \in (-1,1)\,.
\end{align*}
This yields
\begin{align}
 g(v) =   \big (1+\sigma^2 (\partial_z v )^2 \big)^{3/2} \, \frac{\vert \partial_r \phi_v(\, \cdot \,,1) \vert^2}{(1-v)^2} \,, \qquad v \in S(\kappa)\,. \label{g1}
\end{align}
Moreover, as the second preliminary observation, we note that
\begin{align}
 \big \Vert \partial_r \phi_v  ( \, \cdot \,, 1 ) \big  \Vert_{W^{1/2}_2(-1,1)} \leq C( \kappa)\,, \qquad v \in S(\kappa)\,. \label{LipschitzRHS.eq2}
 \end{align}
Indeed, since $\phi_v$ belongs to $W_2^2(\Omega)$, the trace theorem \cite[Theorem 1.5.1.2]{Grisvard85} yields 
$$\Vert \partial_r \phi_v( \,\cdot \,, 1) \Vert_{W_2^{1/2}(-1,1)} \leq C\, \Vert \phi_v \Vert_{W_2^2(\Omega)}\,, \qquad v \in S(\kappa)\,, $$
for some constant $C > 0$ independent of $v $. In combination with the fact that $\phi_v$ is uniformly bounded on $S(\kappa)$ due to \eqref{Estimatephiv}, estimate \eqref{LipschitzRHS.eq2} then follows.\\
{\bf(ii)} We deduce from the representation of $g$ in \eqref{g1} and
\begin{align*}
W^1_q(-1,1) \cdot W^{1/2}_2(-1,1) \cdot W^{1/2}_2(-1,1) \hookrightarrow W^\nu_2(-1,1)\,,
\end{align*}
due to the Multiplication Theorem \ref{multthm}, that 
\begin{align*}
\Vert g(v) \Vert_{W^\nu_{2,D}} &\leq C \, \left \Vert \frac{\big (1+\sigma^2 (\partial_z v)^2 \big)^{3/2}}{(1-v)^2} \right \Vert_{W^1_q(-1,1)} \, \big \Vert \partial_r \phi_v(\,\cdot \,, 1) \big \Vert_{W_2^{1/2}(-1,1)}^2 \\
&\leq C(\kappa) 
\end{align*}
for $v \in S(\kappa)$. Here, the last inequality follows from \eqref{LipschitzRHS.eq2}. Consequently, $g$ maps $S(\kappa)$ to $W^\nu_{2,D}(-1,1)$ and is bounded.
\\{\bf(iii)} We present the main part of the proof. Namely, we derive the stated Lipschitz continuity of $g$ based on \eqref{g1}. To this end, we write
\begin{align*}
 \Vert &g(v) - g(w) \Vert_{W^{\nu}_{2,D}(-1,1)} \\
 &\leq  \left \Vert \frac{(1+\sigma^2 w_z^2)^{3/2}}{(1-w)^2}\, \Big ( \vert \partial_r \phi_v(\,\cdot\,,1)\vert^2-\vert\partial_r \phi_w(\,\cdot\,,1)\vert^2 \Big)\right \Vert_{W_{2,D}^\nu(-1,1)} \\
 &\ \ \ + \left \Vert  (1+\sigma^2w_z^2)^{3/2} \,\left ( \frac{1}{(1-w)^2} -\frac{1}{(1-v)^2}\right)\, \vert \partial_r \phi_v(\,\cdot \,,1)\vert^2 \right \Vert_{W^\nu_{2,D}(-1,1)} \\
 & \ \ \ + \left \Vert \Big ((1+\sigma^2v_z^2)^{3/2}-(1+\sigma^2w_z^2)^{3/2} \Big) \, \frac{1}{(1-v)^2}\, \vert \partial_r \phi_v(\, \cdot \,, 1) \vert^2   \right \Vert_{W^\nu_{2,D}(-1,1)}\\
 &=:I+II+III\,,
\end{align*}
and estimate each part separately:\\
{\it For $I$:} We let $\alpha \in (\xi , 1/2-\nu)$, and write
\begin{align*}
I=\left \Vert \frac{(1+\sigma^2 w_z^2)^{3/2}}{(1-w)^2}\, \Big (  \partial_r \phi_v(\,\cdot\,,1)+\partial_r \phi_w(\,\cdot\,,1) \Big)\Big (  \partial_r \phi_v(\,\cdot\,,1)-\partial_r \phi_w(\,\cdot\,,1) \Big)\right \Vert_{W_{2,D}^\nu(-1,1)}\,.
\end{align*}
From
\begin{align*}
 W^1_q(-1,1)\cdot W^{1/2}_2(-1,1) \cdot W_2^{1/2-\alpha} (-1,1) \hookrightarrow W^\nu_2(-1,1),
\end{align*}
which holds thanks to the Multiplication Theorem \ref{multthm}, we deduce that
\begin{align*}
 I &\leq \left \Vert \frac{(1+\sigma^2 w_z^2)^{3/2}}{(1-w)^2}\right \Vert_{W^1_q(-1,1)}\, \big \Vert \partial_r \phi_v(\,\cdot \,, 1) + \partial_r \phi_w(\, \cdot \,,1) \big \Vert_{W^{1/2}_2(-1,1)} \\ 
&\hspace{4.5cm} \times\big \Vert \partial_r \phi_v(\,\cdot \,,1) - \partial_r \phi_w(\,\cdot\,,1) \big \Vert_{W^{1/2-\alpha}_2(-1,1)}
 \\
 &\leq C( \kappa) \, \big \Vert \partial_r \phi_v(\,\cdot \,,1) - \partial_r \phi_w(\,\cdot\,,1) \big \Vert_{W^{1/2-\alpha}_2(-1,1)}\\
  &\leq C( \kappa)\,  \Vert \partial_r \phi_v  -\partial_r \phi_w \Vert_{W^{1-\alpha}_2(\Omega)}\,,\qquad v,w \in S(\kappa)\,.
\end{align*}
In addition to the Multiplication Theorem, we applied \eqref{LipschitzRHS.eq2}, the fact that $W^1_q(-1,1)$ is a Banach algebra and the chain rule to derive the second estimate, while the third estimate follows from properties of the trace, see \cite[Theorem 1.5.1.2]{Grisvard85}.
Now using continuity of differentiation between fractional Sobolev spaces due to \cite[Theorem 1.4.4.6]{Grisvard85} (which is applicable as $1-\alpha \neq 1/2$) and subsequently Lemma \ref{EstimateSolution}, we conclude that
\begin{align*}
I&\leq C( \kappa)\ \Vert \phi_v  - \phi_w \Vert_{W^{2-\alpha}_2(\Omega)} \\
&\leq C( \kappa)\Vert v-w \Vert_{W_q^{2-\xi}(-1,1)} \,, \qquad v,w \in S(\kappa)\,.
\end{align*}
For $II$: We estimate
\begin{align*}
II &\leq \big \Vert (1+\sigma^2w_z^2)^{3/2} \big \Vert_{W^1_q(-1,1)} \,\left \Vert \frac{1}{(1-w)^2} - \frac{1}{(1-v)^2} \right \Vert_{W^1_q(-1,1)} \, \Vert \partial_r \phi_v(\,\cdot\,,1) \Vert^2_{W^{1/2}_2(-1,1)} \\
&\leq C( \kappa)\,\left \Vert \frac{1}{(1-w)^2} - \frac{1}{(1-v)^2}\right \Vert_{W^1_q(-1,1)} \,, \qquad v,w \in S(\kappa)\,,
\end{align*}
where we use that $W^1_q(-1,1)$ is a Banach algebra and
\begin{align*}
W^1_q(-1,1) \cdot W^{1/2}_2(-1,1) \cdot W^{1/2}_2(-1,1) \hookrightarrow W^\nu_2(-1,1)\,,
\end{align*}
thanks to the Multiplication Theorem \ref{multthm}. 
Writing 
\begin{align*}
\frac{1}{(1-w)^2} - \frac{1}{(1-v)^2} = \frac{2-w-v}{(1-w)^2\,(1-v)^2}\,(w-v) \,,
\end{align*}
and using once more that $W^1_q(-1,1)$ is an algebra, we deduce further that
\begin{align*}
II &\leq C( \kappa) \, \Vert v-w \Vert_{W^1_q(-1,1)} \\
&\leq C( \kappa)\,\Vert v-w \Vert_{W_q^{2-\xi}(-1,1)} \,, \qquad v,w \in S(\kappa).
\end{align*}
For $III$: We rewrite 
\begin{align*}
(1+\sigma^2v_z^2)^{3/2}-(1+\sigma^2w_z^2)^{3/2} &= (1+\sigma^2v_z^2)^{1/2} \left (  (1+\sigma^2v_z^2)-(1+\sigma^2 w_z^2) \right) \\
&\quad \ \ + (1+\sigma^2 w_z^2) \left ((1+\sigma^2 v_z^2)^{1/2} -(1+\sigma^2 w_z^2 )^{1/2} \right )  \\
&=  r(v,w)\,(v_z+w_z)(v_z-w_z)
\end{align*}
with
\begin{align*}
r(v,w):= \sigma^2 \bigg ((1+\sigma^2 v_z^2)^{1/2} + \frac{(1+\sigma^2 w_z^2)}{\sqrt{1+\sigma^2 v_z^2} +\sqrt{1+\sigma^2 w_z^2}} \bigg ) \in W^1_q(-1,1)\,.
\end{align*}
Then, we estimate $III$ by 
\begin{align*}
III&= \left\Vert  \frac{r(v,w)(v_z+w_z)}{(1-v)^2}\,(v_z-w_z)\, \vert \partial_r \phi_v(\, \cdot \,, 1) \vert^2   \right \Vert_{W^\nu_{2,D}(-1,1)} \\
&\leq C( \kappa) \Vert v_z-w_z \Vert_{W^{1-\xi}_q(-1,1)} \\
&\leq C( \kappa) \Vert v-w \Vert_{W^{2-\xi}_{q}(-1,1)}
\end{align*}
using 
\begin{align*}
W^1_q(-1,1) \cdot W^{1-\xi}_q(-1,1)  \cdot W^{1/2}_2(-1,1) \cdot W^{1/2}_2(-1,1) \hookrightarrow W^\nu_{2,D}(-1,1)
\end{align*}
due to the Multiplication Theorem \ref{multthm}. .
Combining the estimates for $I$-$III$ yields \eqref{eqEstimateRHS}. \\
{\bf(iv)} The second estimate \eqref{eqEstimateRHSb} follows directly: 
\begin{align*}
 \left \Vert \frac{1}{v+1} - \frac{1}{w+1} \right \Vert_{W^\nu_{2,D}(-1,1)} &\leq C \, \left \Vert \frac{w-v}{(v+1)(w+1)} \right \Vert_{W^1_q(-1,1)} \\
 &\leq C(\kappa) \Vert w-v \Vert_{W^{2-\xi}_{q,D}(-1,1)} \,, \qquad v,w \in S(\kappa)\,.
\end{align*}
\end{proof}

Thanks to Sobolev's embedding theorem, we have the following $L_q$-$L_q$-version of Proposition \ref{LipschitzRHS}: 

\begin{cor}\label{LipschitzRHSb}
 Let $q \in (2,\infty)$, $\kappa \in (0,1)$ and $\lambda,\sigma > 0$. For $\xi \in [0,1/q)$ and $2\mu \in [0, 1/q-\xi)$, there exists a constant $C_{7}( \kappa) >0$ such that the map
$$[v \mapsto g(v)]\,: \quad S(\kappa)\rightarrow W^{2\mu}_{q,D}(-1,1)$$
is bounded by $C_7(\kappa)$ and
\begin{align*}
\Vert g (v) &- g(w) \Vert_{L_{q}(-1,1)} \leq  \, \frac{C_7( \kappa)}{2\lambda} \,\Vert v-w \Vert_{W_{q,D}^{2-\xi}(-1,1)} \,,
\end{align*}
as well as
\begin{align*}
 \bigg \Vert \frac{1}{v+1} -\frac{1}{w+1} \bigg \Vert_{L_q(-1,1)} \leq  \frac{C_7( \kappa)}{2}\,\Vert v-w \Vert_{W_{q,D}^{2-\xi}(-1,1)} \,, \qquad v,w \in S(\kappa)\,. 
\end{align*}
\end{cor}

\begin{proof}
Since
\begin{align*}
 2 \mu + 1/2-1/q < 1/q - \xi +1/2 -1/q=1/2-\xi\,,
\end{align*}
we can fix $\nu \in \big ( 2 \mu + 1/2-1/q\,, \,1/2-\xi\big )$. While Sobolev's embedding theorem ensures that
\begin{align*}
 W^\nu_{2,D}(-1,1) \hookrightarrow W^{2\mu}_{q,D}(-1,1)\,,
\end{align*}
the choice of $\xi$ and $\nu$ is compatible with Proposition \ref{LipschitzRHS}.
\end{proof}

Note that Proposition \ref{LipschitzRHS} and the corresponding Corollary \ref{LipschitzRHSb} establish Lipschitz continuity of $g$ with respect to a weaker norm than the $\Vert \cdot \Vert_{W^2_{q,D}(-1,1)}$-norm, which will be essential to prove local existence in the quasilinear setting.

\section{Coupled System}\label{21123}
We turn to the proof of the first two main results, Theorem \ref{localexistence} and Corollary \ref{globex}, i.e. well-posedness as well as the global existence criterion for the coupled free boundary problem \eqref{filmdimensionless2}-\eqref{RBPsi42}. As already mentioned, we reinterpret the system \eqref{filmdimensionless2}-\eqref{RBPsi42} as the following single quasilinear parabolic equation for the film deflection 
\begin{align}
\partial_t u - \sigma \partial_z \mathrm{arctan}(\sigma \partial_z u)=G(u) \label{311023}
\end{align}
with Lipschitz continuous and non-local right-hand side $[u \mapsto G(u)]$ given by 
\begin{align}
G(u):= -\frac{1}{u+1}+\lambda g(u)\,. \label{Gdefinition} 
\end{align} 
Our proof follows \cite{ELW15}.
\\

Before we start, we introduce some notations: Let $q \in (2, \infty)$ and $\xi \in (0, 1/q^\prime)$ where $q^\prime$ denotes the dual exponent of $q$. For $\kappa \in (0,1)$, we put 
\begin{align*}
 Z(\kappa):= \big \lbrace v \in W_q^{2-\xi} (-1,1) \, &\big \vert \, \Vert v \Vert_{W_q^{2-\xi}(-1,1)} \leq 1/\kappa \,, \ -1+\kappa \leq v(z) \leq 1-\kappa \, \big \rbrace 
\end{align*}
and define
\begin{align*}
 B(v)w := - \frac{\sigma^2}{(1+ \sigma^2 v_z^2)} w_{zz} \,, 
\qquad w \in W^2_{q,D}(-1,1)\,,
\end{align*}
for $v \in Z(\kappa)$, where the connection between $[v \mapsto B(v)]$ and \eqref{311023} is given via
$$B(u)u= -\frac{\sigma^2 u_{zz}}{(1+\sigma^2 u_z^2)} = - \sigma \partial_z \arctan(\sigma \partial_z u)\,, \qquad u \in W^2_{q,D}(-1,1)\footnote{It is clear that $B(u)$ is also defined in this case.}\,,$$
which is the second order operator occuring on the left-hand side of \eqref{311023}.
Furthermore, the choice of $\xi$ and Sobolev's embedding theorem ensure that $-\displaystyle\frac{\sigma^2}{(1+ \sigma^2 v_z^2)}\in C\big ([-1,1]\big )$ so that each $B(v)$ is uniformly elliptic. \\
 
In the next two lemmata, we establish properties of $[v \mapsto B(v)]$. More precisely, we show that $[v \mapsto B(v)]$ is globally Lipschitz continuous on $Z(\kappa)$ and that each $-B(v)$ generates an analytic semigroup satisfying uniform estimates for $v\in Z(\kappa)$.\\

\begin{lem}\label{GeneratorSemigroup1}
Let $q \in (2, \infty)$, $\kappa \in (0,1)$ and $\xi \in (0,1/q^\prime)$. 
Then, there exists a constant $l( \kappa)$ such that 
\begin{align*}
 \Vert B(w)-B(v) \Vert_{\mathcal{L}   (W^2_{q,D}(-1,1),L_q(-1,1) )} \leq l(\kappa) \Vert w-v \Vert_{W_{q,D}^{2 - \xi} (-1,1)} 
\end{align*}
for $v,w\in Z(\kappa)$.
\end{lem}

\begin{proof}
The statement follows from
 \begin{align*}
  \Vert B(w) &- B(v) \Vert_{\mathcal{L} (W^2_{q,D}(-1,1),L_q(-1,1))}\\
  &\leq \sigma^2\,\left \Vert \frac{1}{(1+\sigma^2w_z^2)}  - \frac{1}{(1+ \sigma^2 
v_z^2)} \right \Vert_\infty \\
  &\leq \sigma^4 \left \Vert \frac{1}{(1+\sigma^2 w_z^2)(1+\sigma^2 v_z^2)} \right \Vert_\infty \, \Vert w_z+v_z \Vert_\infty\, \Vert w_z-v_z \Vert_\infty\\
 &\leq l( \kappa) \Vert w-v \Vert_{W^{2-\xi}_{q,D}(-1,1)} \,,
 \end{align*}
 where we made use of $Z(\kappa)$ being continuously embedded and bounded in $C^1 \big([-1,1] \big )$ due to Sobolev's embedding theorem.
\end{proof}

\begin{lem}\label{GeneratorSemigroup2}
Let $q \in (2, \infty)$, $\kappa \in (0,1)$ and $\xi \in (0,1/q^\prime)$. Moreover, let $\omega >0$ be fixed. 
Then, there is a constant $k:=k(  \kappa) \geq 1$ such that for each $v \in Z(\kappa)$ one has
\begin{align*} B(v) \in \mathcal{H}\big(W^2_{q,D}(-1,1), L_q(-1,1), k, \omega \big )\,.\end{align*}
\end{lem}

\begin{proof}
In \cite[Remark I.1.2.1\,(a)]{AmannLQPP} a criterion for $B(v)$ to belong to one of the quantitative versions of $\mathcal{H}\big (W^2_{q,D}(-1,1),L_q(-1,1) \big )$, introduced in \eqref{QSE}, is given which reads as follows:\\

Assume that there are constants $C_i( \kappa)>0$ for $i=8,9$ such that for all $v \in Z(\kappa)$ one has: 
\begin{itemize}
\item[{\bf(i)}] $\big \Vert B(v) \big \Vert_{\mathcal{L}(W^2_{q,D}(-1,1),L_q(-1,1))} \leq C_8( \kappa)$\,,
\item[{\bf(ii)}] $[\mathrm{Re} \mu \geq \omega] \in \rho(-B(v))$ and
\begin{align*}
 \big \Vert [\mu+B(v)]^{-1} \big \Vert_{\mathcal{L}(L_q(-1,1))} \leq \frac{C_8( \kappa)}{\vert \mu \vert} \,, \qquad \mathrm{Re} \mu \geq \omega \,,
\end{align*}
\item[{\bf(iii)}] $\big \Vert [\omega +B(v)]^{-1} \big \Vert_{\mathcal{L}(L_q(-1,1),W_{q,D}^2(-1,1))} \leq C_9( \kappa)$.
\end{itemize}
Then, the assertion of the lemma follows from \cite[Remark I.1.2.1\,(a)]{AmannLQPP}. \\

Thus, we only have to check (i)--(iii). We first define $V:= (1+ \sigma^2 v_z^2)/\sigma^2$ so that $B(v)w= - 1/V\, w_{zz}$. Then, 
\begin{align}1/\sigma^2 \leq V\leq C_{10}( \kappa)\,, \qquad v \in Z(\kappa)\,,\label{vvv} \end{align}
and the differential operator $B(v)$ satisfies 
\begin{align*}
\big \Vert B(v) \big \Vert_{\mathcal{L}(W^2_{q,D}(-1,1), L_q(-1,1))} \leq \sigma^2 \,, \qquad  v\in Z(\kappa)\,,
\end{align*}
which is condition (i). \\

Next, we check condition (ii).
To this end, we note that $B(v)$ is uniformly elliptic. Consequently, for $f \in L_q(-1,1)$, the equation
\begin{align*}
 \begin{cases}
B(v) u &= f\,, \\
u(\pm1)&=0
 \end{cases} 
\end{align*}
is uniquely solvable in $W^2_{q,D}(-1,1)$ with $B(v)^{-1} \in \mathcal{L} \big(L_q(-1,1), W^2_{q,D}(-1,1)\big )$ due to \cite[Theorem 9.15, Lemma 9.17]{GT98}.
It follows from the Theorem of Rellich-Kondrachov that $B(v)^{-1} \in 
\mathcal{L} \big (L_q(-1,1) \big )$ is compact, and \cite[Theorem 6.29]{Kato95} implies that the spectrum $\sigma(-B(v))$ consists only of eigenvalues.
Now we fix an eigenvalue $\mu$ of $-B(v)$ and a corresponding eigenfunction $\varphi \in W^2_{q,D}\big ((-1,1),\mathbb{C} \big)$. Testing
$$\mu \varphi - \frac{1}{V} \partial_z^2 \varphi =0$$
with $V \,\overline{\varphi} \in W^2_{q^\prime,D}\big ((-1,1),\mathbb{C} \big )$ and using integration by parts yields 
\begin{align*}
 \mu = \frac{-\int_{-1}^1 \vert \partial_z \varphi \vert^2 \, \mathrm{d}z}{\int_{-1}^1 V\, \vert \varphi \vert^2 \, \mathrm{d} z} <0
\end{align*}
so that
\begin{align*}
 \big [ \mathrm{Re}\mu > 0 ] \subset \rho (-B(v)) \,, \qquad v \in Z(\kappa) \,.
\end{align*}
Next, let $u \in W_{q,D}^2(-1,1)$ be the unique solution to
\begin{align*}
 [\mu +B(v) ] u = f \,, \qquad f \in L_q\big((-1,1),\mathbb{C} \big)\,,
\end{align*}
for $\mu >0$. Testing this equation with $V\, \vert u \vert^{q-2} \overline{u} \in L_{q^\prime}\big((-1,1), \mathbb{C} \big)$ yields -- along the lines of the proof of \cite[Proposition 2.4.2]{LLMP04} -- the resolvent estimate (ii).\\

Finally, we turn to condition (iii). For $v \in Z(\kappa)$ and $u \in W^2_{q,D}(-1,1)$, we find 
\begin{align}
 \Vert u \Vert_{W^2_{q,D}(-1,1)}^q &\leq \Vert u \Vert_{W^1_q(-1,1)}^q + C_{10}( \kappa)^q \big \Vert B(v) u  \big \Vert_{L_q(-1,1)}^q \nonumber \\
 &\leq \frac{1}{2} \Vert u \Vert_{W^2_{q,D}(-1,1)}^q + C\, \Vert u \Vert_{L_q(-1,1)}^q + C_{10}( \kappa)^q \big \Vert [\omega+ B(v)] u  \big \Vert_{L_q(-1,1)}^q \label{3.12.neu}
\end{align}
thanks to \eqref{vvv}, the triangle inequality and Ehrling's lemma. Rearranging \eqref{3.12.neu}, we deduce from (ii) the existence of a constant $C_9( \kappa)>0$ with
\begin{align*}
 \Vert u \Vert_{W^2_{q,D}(-1,1)} \leq C_9( \kappa) \big \Vert [\omega +B(v)]u \big \Vert_{L_q(-1,1)} \,, \quad v \in Z(\kappa)\,, \quad u \in W^2_{q,D}(-1,1)\,,
\end{align*}
which is equivalent to condition (iii). Now everything is proven.

\end{proof}

 If $v$ now depends on $t$, then $-B(v)$ generates a parabolic evolution operator (instead of an analytic semigroup), which satisfies regularity estimates holding uniformly on $Z(\kappa)$. The corresponding result is \cite[Proposition 3.2]{ELW15}.

\begin{prop}\label{SemigroupEstimates}
 Let $q \in (2, \infty)$, $\kappa \in (0,1)$, $\rho \in (0,1)$ and $\xi \in (0, 1/q^\prime)$. For $\tau \in (0,1]$, we define
 \begin{align*}
  \mathcal{V}_\tau(\kappa):= \bigg \lbrace v :[0,\tau] &\rightarrow W^{2-\xi}_{q,D}(-1,1) \, \bigg \vert \, \\
  &\Vert v(t)-v(s) \Vert_{W^{2-\xi}_{q,D}(-1,1)} \leq \vert t-s \vert^\rho\,, \quad  v(t) \in Z(\kappa)\,,\quad s,t\in[0,\tau] \bigg \rbrace\,. 
 \end{align*}
Then, for each $v \in \mathcal{V}_\tau(\kappa)$, there exists a unique parabolic  
$$\big \lbrace U_{B(v)}(t,s) \,\big \vert \, 0 \leq s \leq t \leq \tau \, \big \rbrace$$
possessing $W^2_{q,D}(-1,1)$ as regularity subspace. Moreover, for fixed $2\mu \in (0, 1/q)$, there exists a constant 
$C_{11}(\kappa) \geq 1$ independent of $\tau$ and $v \in \mathcal{V}_\tau(\kappa)$ such that 
\begin{align*}
 \big \Vert U_{B(v)}(t,s) \big \Vert_{\mathcal{L}(W^{2}_{q,D}(-1,1))}+\big (t-s)^{1-\mu}\big \Vert U_{B(v)}(t,s) \Vert_{\mathcal{L}(W^{2\mu}_{q,D}(-1,1), W^{2}_{q,D}(-1,1))} \leq C_{11}(\kappa) 
\end{align*}
for $0 \leq s < t \leq \tau$.
\end{prop}

\begin{proof}
Let $\omega >0$ and put
$$\mathcal{B}:= \Big \lbrace \big [t \mapsto B(v(t)) \big ] \, \Big \vert \, v \in \mathcal{V}_\tau(\kappa) \Big \rbrace.$$
From Lemma \ref{GeneratorSemigroup1} and Lemma \ref{GeneratorSemigroup2}, we deduce that 
$$\mathcal{B} \subset C^\rho \Big (\,[0,\tau] \,,\, \mathcal{H}\big(W^2_{q,D}(-1,1),L_q(-1,1), k, \omega \big) \Big)$$ 
is bounded, which implies that $\mathcal{B}$ satisfies condition \cite[Equation II\,(5.0.1)]{AmannLQPP}. Here, $k=k( \kappa)\geq 1$ is the same as in Lemma \ref{GeneratorSemigroup2}. Since condition \cite[Equation II\,(5.0.1)]{AmannLQPP} is satisfied,
we can use the uniform estimates for parabolic evolution operators from \cite[Section II.5]{AmannLQPP}.
More precisely, the statement follows from \cite[Theorem II.5.1.1, Lemma II.5.1.3]{AmannLQPP} and the identification of interpolation spaces as fractional Sobolev spaces with Dirichlet boundary conditions based on \cite[Theorem 5.2]{Amann93}.
The latter originates from \cite{Grisvard67,Seeley72}.
\end{proof}

\begin{bem}\label{wvor}
The above proof ensures that the uniform estimates from \cite[Section II.5]{AmannLQPP} hold true. Together with the regularity estimates for the non-local operator $[u \mapsto G(u)]$ defined in \eqref{Gdefinition}, see Corollary \ref{LipschitzRHSb}, they form the basis for the upcoming fixed point argument. \\
\end{bem}

We are now in a position to establish local well-posedness for the free boundary problem \eqref{filmdimensionless2}-\eqref{RBPsi42}:\\

{\bf Proof of  Theorem \ref{localexistence}.}
It suffices to show the existence of a unique local solution $u$ to \eqref{311023}, which may subsequently be extended to a unique maximal solution. We want to apply Banach's fixed point theorem:\\
{\bf(i)} {\it Choice of a complete metric space}: Fix $\kappa > 0$ with 
$$u_0 \in S(2 \kappa) \cap Z(2 \kappa)$$
as well as 
$$\xi \in (0,1/q) \,, \qquad \rho \in (0,\xi/4)\,, \qquad 2 \mu \in (0,1/q-\xi)\,, \qquad \tau \in (0,1]\,.$$
Here, we recall that $u_0 \in S(2 \kappa)$ is equivalent to 
 $$ \Vert u_0 \Vert_{W^2_{q,D}(-1,1)} \leq \frac{1}{2\kappa} \,, \qquad 1-2\kappa \geq u_0 \geq -1 +2 \kappa \,,$$
 while $u_0 \in Z(2 \kappa)$ is equivalent to 
 $$ \Vert u_0 \Vert_{W^{2-\xi}_{q,D}(-1,1)} \leq \frac{1}{2\kappa} \,, \qquad 1-2\kappa \geq u_0 \geq -1 +2 \kappa\,, $$
 where different norms are used due to the fact that the analysis of the right-hand side of \eqref{311023} requires control of the $W^2_q$-norm, while the arguments from \cite[Section II.5]{AmannLQPP} only apply for slightly weaker norms. Moreover, by Proposition \ref{SemigroupEstimates}, we find $C_{11}(\kappa) \geq 1$ independent of $\tau$ such that 
 \begin{align}
 \big \Vert U_{B(v)}(t,s) \big \Vert_{\mathcal{L}(W^{2}_{q,D}(-1,1))}+\big (t-s)^{1-\mu}\big \Vert U_{B(v)}(t,s) \Vert_{\mathcal{L}(W^{2\mu}_{q,D}(-1,1), W^{2}_{q,D}(-1,1))} \leq C_{11}(\kappa)   \label{PPPP}
\end{align}
for each $v \in \mathcal{V}_\tau(\kappa)$ and $0 \leq s < t \leq \tau$. Now, we put $\tilde{\kappa} := \displaystyle\frac{\kappa}{C_{11}(\kappa)} \leq \kappa$
 and define 
  \begin{align*}
  \mathcal{V}_\tau(\kappa, \tilde{\kappa}):= \bigg \lbrace v :\,&[0,\tau] \rightarrow W^{2}_{q,D}(-1,1) \, \bigg \vert \, \\
  &\Vert v(t)-v(s) \Vert_{W^{2-\xi}_{q,D}(-1,1)} \leq \vert t-s \vert^\rho\,, \ \  v(t) \in S(\tilde{\kappa}) \cap Z(\kappa)\,,\ \ s,t\in[0,\tau] \bigg \rbrace 
 \end{align*}
 with $\mathcal{V}_\tau(\kappa, \tilde{\kappa}) \subset \mathcal{V}_\tau(\kappa)$. Thanks to the Theorem of Eberlein-Smulyan, $\mathcal{V}_\tau(\kappa, \tilde{\kappa})$, equipped with the metric
 $$d(v,w):= \sup_{t \in [0,\tau]} \Vert v(t) -w(t) \Vert_{W^{2-\xi}_{q,D}(-1,1)}\,,$$
 is a complete metric space. \\
 {\bf(ii)} {\it Definition of the map $\Lambda$}: Recall from \eqref{Gdefinition} that we use the abbreviation
 $$G(v(t))=\frac{-1}{1+v(t)} + \lambda g(v)(t) \,, \qquad v \in \mathcal{V}_{\tau}(\kappa,\tilde{\kappa}) \,, \qquad t \in [0,\tau]\,,$$
 for the right-hand side of \eqref{311023}, and note that
 $$[t \mapsto G(v(t))] \in C^\rho \big ([0,\tau], L_q(-1,1) \big )$$
 due to Corollary \ref{LipschitzRHSb}. Hence, thanks to \cite[Theorem II.1.2.1, Remark II.2.1.2\,(b)]{AmannLQPP}, the variation-of-constant-formula
 \begin{align*}
  \Lambda(v)(t):= U_{B(v)} (t,0)u_0 + \int_0^t U_{B(v)}(t,s) G(v(s)) \,\mathrm{d} s \,, \qquad t \in [0,\tau]\,,
 \end{align*}
defines for each $v \in \mathcal{V}_\tau(\kappa,\tilde{\kappa})$ the unique solution 
$$\Lambda(v) \in C^1 \big ([0,\tau],L_q(-1,1) \big ) \cap C \big ([0,\tau], W^2_{q,D}(-1,1) \big )$$
to the linear problem
\begin{align*}
  \partial_t u + B(v) u = G(v) \,, \qquad u(0)=u_0 \,.
\end{align*}
It remains to adjust $\tau \in (0,1]$ such that the map $\Lambda$ possesses further properties:\\
{\bf(iii)} {\it $\Lambda$ is a self-mapping}: It follows from \cite[Theorem II.5.3.1]{AmannLQPP} (with $\alpha=1-\xi/2 +2 \rho$ and $\beta = 1-\xi/2$) that 
\begin{align*}
  \Vert \Lambda(v)(t)&-\Lambda(v)(s) \Vert_{W^{2-\xi}_{q,D}(-1,1)}\\
  &\leq C_{12}(\kappa)\,\vert t-s \vert^{2 \rho} \left ( \Vert u_0 \Vert_{W_{q,D}^{2-\xi+4\rho}(-1,1)} + \Vert G(v(t))\Vert_{L_\infty((0,t),L_q(-1,1))} \right)\\
  &\leq C_{13}(\kappa) \left(\frac{1}{2\kappa} +C_7(\tilde{\kappa}) \right) \, \tau^{\rho} \, \vert t-s\vert^{\rho} \qquad v \in \mathcal{V}_\tau(\kappa,\tilde{\kappa}) \,, \ s,t \in [0,\tau]\,,
\end{align*}
where we additionally used the choice of $\kappa$ and Corollary \ref{LipschitzRHSb}. Making $\tau$ smaller, if necessary, we find, for arbitrary $v \in \mathcal{V}_\tau(\kappa,\tilde{\kappa})$ and $s,t \in [0,\tau]$, that
\begin{align}
   \Vert \Lambda(v)(t)-\Lambda(v)(s) \Vert_{W^{2-\xi}_{q,D}(-1,1)} \leq \vert t-s \vert^{\rho} \,.   \label{sm1}
\end{align}
Next, the triangle inequality and \eqref{sm1} imply that 
\begin{align}
 \Vert \Lambda(v)(t) \Vert_{W^{2-\xi}_{q,D}(-1,1)} &\leq \Vert \Lambda(v)(t)-\Lambda(v)(0) \Vert_{W^{2-\xi}_{q,D}(-1,1)} + \Vert u_0 \Vert_{W^{2-\xi}_{q,D}(-1,1)} \nonumber \\
 &\leq \tau^{\rho} + \frac{1}{2\kappa}  \,, \label{sm2}
\end{align}
while \eqref{sm1} combined with Sobolev's embedding theorem gives 
\begin{align}
 \Lambda(v)(t) &\leq u_0 + \Vert \Lambda(v)(t) - \Lambda(v)(0) \Vert_\infty \nonumber \\
&\leq 1- 2 \kappa + C \, \Vert \Lambda(v)(t)-\Lambda(v)(0) \Vert_{W_{q,D}^{2-\xi}(-1,1)} \nonumber \\
& \leq 1-2 \kappa +C\,\tau^{\rho}\,.\label{sm3}
\end{align}
A similar argument yields 
\begin{align}
\Lambda(v)(t) \geq -1 +2 \kappa -C \,\tau^\rho \,.\label{sm4}
\end{align}
Moreover, we have
\begin{align}
\Vert \Lambda (v)(t) \Vert_{W^2_{q,D}(-1,1)} &\leq C_{11}(\kappa) \Vert u_0 \Vert_{W^2_{q,D}(-1,1)} +C_{11}(\kappa) \int_{0}^t (t-s)^{\mu-1} \Vert G(v(s)) \Vert_{W^{2\mu}_{q,D}(-1,1)} \, \mathrm{d} s \nonumber \\
&\leq \frac{C_{11}(\kappa)}{2 \kappa} + C_{11}(\kappa)C_7(\tilde{\kappa}) \int_0^t s^{\mu-1} \, \mathrm{d} s \nonumber \\
&\leq \frac{1}{2 \tilde{\kappa}} + C_{11}(\kappa)C_7(\tilde{\kappa})\,\frac{\tau^\mu}{\mu}  \,, \label{sm5}
\end{align}
where we applied \eqref{PPPP} for the first inequality and Corollary \ref{LipschitzRHSb} for the second one, while the last inequality follows from the choice of $\tilde{\kappa}$. Note that in \eqref{sm5} the role of $\tilde{\kappa}$ becomes clear as we can only show that
$\Vert U_{B(v)}(t,0)u_0 \Vert_{W^2_{q,D}(-1,1)}$ is bounded, but have no possibility to adjust the bound to be smaller than $\frac{1}{2\kappa}$.
Making $\tau \in (0,1]$ smaller, if necessary, equations \eqref{sm1}-\eqref{sm5} imply that $\Lambda$ maps $\mathcal{V}_\tau(\kappa,\tilde{\kappa})$
into itself.\\
{\bf(iv)} {\it $\Lambda$ is a contraction}: Finally, for $v,w \in \mathcal{V}_\tau(\kappa,\tilde{\kappa})$ and $t \in [0,\tau]$, it follows from \cite[Theorem II.5.2.1]{AmannLQPP} (with $\alpha =1$, $\beta=1-\xi/2$ and $\gamma= \mu$) that 
\begin{align*}
 \Vert \Lambda&(v)(t)-\Lambda(w)(t) \Vert_{W^{2-\xi}_{q,D}(-1,1)} \\
 &\leq C_{14}(\kappa) \,\tau^{\xi/2} \, \bigg ( \Vert G(v) - G(w) \Vert_{L_\infty((0,t),L_q(-1,1))}
 \\ &\quad \ \ +\Vert B(v)-B(w) \Vert_{C([0,t], \mathcal{L}(W^2_{q,D},L_q)} \cdot \left (\Vert u_0 \Vert_{W^2_{q,D}(-1,1)} + \Vert G(v) \Vert_{L_\infty((0,t),W^{2\mu}_{q,D}(-1,1))} \right)\bigg ) \\
 &\leq C_{14}(\kappa) \, \tau^{\xi/2} \left( C_{7}(\tilde{\kappa}) + \ell(\kappa) \Big(\displaystyle\frac{1}{2\kappa} + C_7(\tilde{\kappa}) \Big) \right) \, d(v,w)\,. 
\end{align*}
Here, we have also applied Corollary \ref{LipschitzRHSb} and Lemma \ref{GeneratorSemigroup1} for the second inequality. Making $\tau \in (0,1]$ smaller, if necessary, and taking the supremum over $t \in [0,\tau]$, we find
\begin{align*}
 d\big (\Lambda(v), \Lambda(w) \big ) \leq \frac{1}{2} \, d(v,w) \,, \qquad v,w \in \mathcal{V}_\tau(\kappa,\tilde{\kappa})\,,
\end{align*}
i.e. $\Lambda$ is a contraction. \\

In view of (i)-(iv), Banach's fixed point argument yields the local existence of a unique solution $u \in C^1 \big ([0,\tau],L_q(-1,1) \big ) \cap C \big ( [0,\tau],W^2_{q,D}(-1,1) \big)$ while for fixed time $t \in [0,\tau)$ the transformed electrostatic potential $\phi_{u(t)}$ belongs to $W^2_2(\Omega)$ with $\Omega=(-1,1)\times(1,2)$, see \eqref{defphiv}, which is equivalent to $\psi_{u(t)} \in W^2_2 \big (\Omega(u(t))\big)$. Hence, everything is proven.
\qed \\

{\bf Proof of Corollary \ref{globex}.}
Since $\tau$ in the above fixed point argument only depends on $\kappa$ and $\tilde{\kappa}$ which itself only depends on $\kappa$, the statement follows easily by a contradiction argument.
\qed \\

Another consequence of the uniqueness of solutions is the following result on symmetry: 

\begin{cor}\label{Symmetry}
If the initial value $u_0$ is even, i.e. $u_0(z)=u_0(-z)$, then the unique maximal solution $u$ from Theorem \ref{localexistence} and the corresponding electrostatic potential $\psi_u$ are even with respect to $z$ at each time $t \in [0,T_{max})$.
\end{cor}

\begin{proof}
First, let $v \in S(\kappa)$ for some $\kappa >0$ and define $\tilde{v}(z):= v(-z)$. Then, the unique solvability of the electrostatic problem implies that $\psi_{\tilde{v}}(z,r)=\psi_{v}(-z,r)$ for all $(z,r) \in \Omega(\tilde{v})$
and consequently $g(\tilde{v})(z)=g(v)(-z)$ for $z \in (-1,1)$ by definition of the electrostatic force in \eqref{filmdimensionless0}. Now, if the initial value $u_0$ in Theorem \ref{localexistence} is even, i.e. $u_0(z)=u_0(-z)$, then the uniqueness of solutions implies that the maximal solution
$u$ is even in $z$, too. In particular, $\tilde{u}(t) =u(t)$ for each $t \in [0,T_{max})$ and consequently $\psi_{u(t)}(z,r)=\psi_{u(t)}(-z,r)$.
\end{proof}

\section{Non-Existence of Global Solutions for Large Voltages}\label{section: Non-Existence}

The goal of the current section is to prove Theorem \ref{MainResultGlobalNonexistence}, i.e. the non-existence of global solutions for $\sigma \geq \sigma_{crit}$, $u_0 \geq u_{cat}$ and $\lambda$ above a critical value $\lambda_{crit}$. Here, $u_{cat}$ is the catenoid given by \eqref{catenoid}. \\

As a starting point, let us note that the parabolic comparison principle is, at least to some extend, applicable to \eqref{filmdimensionless2}-\eqref{RBPsi42}. More precisely, since the non-local electrostatic force $g(u)$ is always positive, one can show the following:

\begin{prop}\label{PCP}
If $u_0 \geq u_{cat}$, then $u(t) \geq u_{cat}$ for all $t\in [0,T_{max})$.
\end{prop}

\begin{proof} 
This follows from an adaptation of \cite[Theorem 9.7]{L96}.
For details, we refer to \cite[Proposition 6.4]{LSS24}.

\end{proof}

Now, for given solution $(u,\psi_u)$ to \eqref{filmdimensionless2}-\eqref{RBPsi42} with initial value $u_0 \geq u_{cat}$, we can consider the energy functional 
\begin{align*}
\mathcal{E}(t)=- \int_{-1}^{1} \mathrm{ln}\big (u(t,z)+1 \big)\,\mathrm{d}z \,, \qquad t \in [0,T_{max})\,, \qquad T_{max}=T_{max}(u_0)\,,
\end{align*}
for which we have to prove that it decreases and satisfies the inequality \eqref{LjapunovEst}. The inequality  will follow from several auxiliary results.

 \begin{lem}\label{Energyequality}
The functional $\mathcal{E}$ belongs to $C^1 \big ([0,T_{max}),\mathbb{R} \big)$ with derivative
\begin{align}
\frac{\mathrm{d}}{\mathrm{d} t} \mathcal{E}(t)= -\int_{-1}^1 \frac{\partial_t u(t,z)}{u(t,z)+1} \, \mathrm{d} z \,.\label{eqGN1}
\end{align}
\end{lem}

\begin{proof} 
This follows from the the mean value theorem together with the regularity $u \in C \big ([0,T],W^2_q(-1,1) \big) \cap C^1 \big ([0,T],L_q(-1,1) \big)$. 
\end{proof}

\begin{figure}[h]
\begin{tikzpicture}
\fill[red!7] (-2.45,1.4) rectangle (2.45,3);
\draw[domain=-2.45:0, smooth, variable=\z, red!7, fill] plot ({\z},{(0.056*\z*\z*\z*\z+3-0.35*\z*\z-1.4)});
\draw[domain=0:2.45, smooth, variable=\z, red!7, fill] plot ({\z},{(0.056*\z*\z*\z*\z+3-0.35*\z*\z-1.4)});
\draw [<-, bend angle=45, bend left]  (0.6,1.52) to (1,1.72) node[right]{$\psi_u=0$};
\draw[domain=-1.5:1.5, smooth, variable=\z, white, fill] plot ({\z},{(0.056*\z*\z*\z*\z+3-0.35*\z*\z-1.4)});
\draw [<-, bend angle=45, bend left]  (0.6,1.52) to (1,1.72) node[right]{$\psi_u=0$};
\draw[domain=-2.2:2.2, smooth, variable=\z, red, thick] plot ({\z},{(0.056*\z*\z*\z*\z+3-0.35*\z*\z-1.4)});
\draw[thick] (-2.45,3) -- (2.45,3) ;
\draw[thick] (-2.45,3) -- (-2.45,1.46) ;
\draw[thick] (2.45,3) -- (2.45,1.46) ;
\draw[domain=-2.45:2.45, smooth, variable=\z, blue, thick] plot ({\z},{cosh(0.56*\z)-0.64});
\draw [<-, bend angle=10, bend right]  (0.4,0.44) to (0.3,0.65);
\draw[thin] (0.38,0.5) to (0.38,0.51) node[above]{\hspace{-5mm}$\psi_{cat}=0$};
\draw[dashed](2.45,0) -- (2.45,3) ;
\draw[dashed] (-2.45,3) -- (-2.45,0) ;
\draw[dashed] (-2.45,3) -- (-2.45,0) ;
\draw[dashed] (-2.45,3) -- (-2.45,0) ;
\draw[->] (-3,0) -- (3,0) node[right] {$z$};
\draw (2.45,0.1) -- (2.45,-0.1) node[below] {$1$};
\draw (-2.45,0.1) -- (-2.45,-0.1) node[below] {$-1$};
\draw[->] (-2.9,-0.1) -- (-2.9,3.2) node[right] {$r$};
\draw (-2.8,1.5) -- (-3,1.5) node[left] {$1$};
\draw (-2.8,3) -- (-3,3) node[left] {$2$};
\node[above] at (0,1.9) {$\Omega(u)$};
\end{tikzpicture}
\caption[Situation in the Proof of Proposition \ref{lowerEstimateRHS}]{The situation in the proof of Proposition \ref{lowerEstimateRHS}: The film deflection $u+1$ (red) lies above the catenoid $u_{cat}+1$ (blue).
Note that the electrostatic potentials $\psi_u$ and $\psi_{cat}$ coincide on the black boundary parts of $\Omega(u)$ and are positive there.}\label{figmaxprin}
\end{figure}
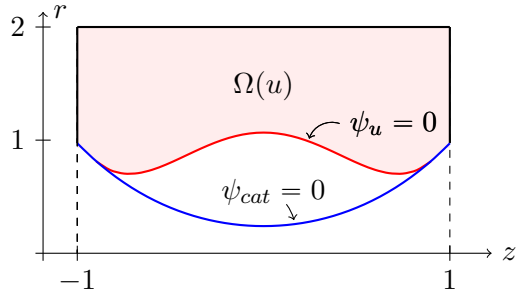

Next, we derive suitable estimates for the right-hand side of \eqref{eqGN1}, and, as the main step, we connect the electrostatic force with the $L_1$-norm of $\partial_z u$. The proof is based on Gauss's theorem and its idea is inspired by \cite{ELW13}.

\begin{prop}\label{lowerEstimateRHS}
There exists a constant $C_{15}(\sigma) >0$ \color{black} (independent of $t$ and $u_0$) such that 
\begin{align*}
\int_{-1}^1 \big ( 1+ (\sigma \partial_z u)^2 \big)^{3/2} \vert \partial_r \psi_u (z,u+1) \vert^2 \,\mathrm{d} z \geq \varepsilon\, C_{15}(\sigma) \color{black} - \varepsilon^2 \int_{-1}^1 \sqrt{1+(\sigma\partial_z u)^2}\, \mathrm{d} z
\end{align*}
for each $t \in [0,T_{max})$ and each $\varepsilon >0$.
\end{prop}

\begin{proof} 
In the following, we fix $t \in [0,T_{max})$ and use the abbreviations $\psi_{u}:= \psi_{u(t)}$ and $\psi_{cat}:=\psi_{u_{cat}}$. The situation is depicted in Figure \ref{figmaxprin}. Since $u$ always stays above the catenoid $u_{cat}$, we can consider the function $f:=\psi_{u} - \psi_{cat}$ in $\Omega(u)$.\\ On the boundary of $\Omega(u)$, this function satisfies 
\begin{align*}
f&=0 \quad \text{on} \quad \lbrace \pm 1\rbrace \times [1,2]\,, \\
f&=0 \quad \text{on} \quad [-1,1] \times \lbrace 2 \rbrace\,, \\
f &\leq 0 \quad \text{on} \quad \mathrm{graph}(u+1)\,.
\end{align*}
Therefore, the maximum principle implies that $f$ attains its maximum on the whole boundary parts $\lbrace \pm 1\rbrace \times [1,2]$ and $[-1,1] \times \lbrace 2 \rbrace$. Hence, the outer normal derivative of $f$ satisfies
\begin{align*}
\partial_\nu f \geq 0 \quad \text{on} \quad \lbrace \pm 1\rbrace \times (1,2) \quad \text{and on} \quad (-1,1) \times \lbrace 2 \rbrace\,,
\end{align*}
which is equivalent to 
\begin{align*}
\partial_\nu \psi_u \geq \partial_\nu \psi_{cat} \quad \text{on} \quad \lbrace \pm 1\rbrace \times (1,2) \quad \text{and on} \quad (-1,1) \times \lbrace 2 \rbrace\,.
\end{align*}
Since $\psi_u$ solves
\begin{align}
0=\mathrm{div} \bigg (  r  \begin{pmatrix} \sigma^2 \partial_z \psi_u \\ \partial_r \psi_u \end{pmatrix} \bigg ) \quad \text{in} \quad \Omega(u)\,, \label{cylindiv}
\end{align}
we deduce from Gauss's theorem and $\partial_z \psi_u = - \partial_z u \, \partial_r \psi_u$ on $\mathrm{graph}(u+1)$ that 
\begin{align}
\int_{-1}^1 (u+1)\big ( 1&+ (\sigma \partial_z u)^2 \big) \partial_r \psi_u (z,u+1) \,\mathrm{d} z \nonumber \\&= - \int_{\mathrm{graph}(u+1)} r\, \begin{pmatrix} \sigma^2 \partial_z \psi_u \\ \partial_r \psi_u \end{pmatrix}\cdot \nu \, \mathrm{d} o(z,r) \nonumber \\
&=\int_{1}^2 \sigma^2\,r \,\partial_\nu \psi_u(-1,r) \,\mathrm{d} r + \int_{-1}^1 2\,  \partial_\nu \psi_u(z,2)  \, \mathrm{d} z + \int_{1}^2  \sigma^2\,r\, \partial_\nu \psi_u(1,r) \,\mathrm{d} r \nonumber \\
&\geq \int_{1}^2 \sigma^2\,r \,\partial_\nu \psi_{cat}(-1,r) \,\mathrm{d}r + \int_{-1}^1 2 \,\partial_\nu \psi_{cat}(z,2)  \, \mathrm{d}z + \int_{1}^2 \sigma^2\,r\, \partial_\nu \psi_{cat}(1,r) \,\mathrm{d} r\nonumber \\
&= - \int_{\mathrm{graph}(u_{cat}+1)} r\, \begin{pmatrix} \sigma^2 \partial_z \psi_{cat} \\ \partial_r \psi_{cat} \end{pmatrix}\cdot \nu  \, \mathrm{d} o(z,r)\nonumber \\
&=\int_{-1}^1 (u_{cat}+1)\big ( 1+ (\sigma \partial_z u_{cat})^2 \big) \partial_r \psi_{cat} (z,u_{cat}+1) \,\mathrm{d} z =:  C_{15}(\sigma)\,. \label{eqGN3}
\end{align}

In the last step, we have used that $\psi_{cat}$ solves \eqref{cylindiv} in $\Omega(u_{cat})$. Next, we show that $C_{15}(\sigma) > 0$: Because $\psi_{cat}$ attains its minimum on the whole $\mathrm{graph}(u_{cat}+1)$, it follows from Hopf's Lemma that $\partial_\nu \psi_{cat} < 0$, and hence 
\begin{align*}
0 &>  \partial_z \psi_{cat}\big(z,u_{cat}(z)+1\big) \partial_z u_{cat}(z) - \partial_r \psi_{cat}\big(z,u_{cat}(z)+1\big)\\
&= - \big (1+\partial_z u_{cat}(z)^2 \big )\, \partial_r \psi_{cat}\big (z,u_{cat}(z)+1 \big ) \,.
\end{align*} 
Consequently, $\partial_r \psi_{cat} > 0$ on $\mathrm{graph}(u_{cat}+1)$ and 
\begin{align*}
\begin{pmatrix} \sigma^2 \partial_z \psi_{cat} \\ \partial_r \psi_{cat} \end{pmatrix}\cdot \nu &= \begin{pmatrix} \sigma^2 \partial_z \psi_{cat} \\ \partial_r \psi_{cat} \end{pmatrix} \cdot \frac{1}{\sqrt{1+\partial_z u_{cat}^2}} \begin{pmatrix}
 \partial_z u_{cat} \\ -1 \end{pmatrix} \\
 &= \frac{-\big(1+(\sigma\partial_z u_{cat})^2\big)}{\sqrt{1+\partial_z u_{cat}^2}}\, \partial_r \psi_{cat} \big (z,u_{cat}(z)+1 \big) <0\,,
\end{align*}
which implies $C_{15}(\sigma) >0$. Now we are ready to finish off the proof: A combination of \eqref{eqGN3} with $u+1 \in (0,2)$ and the weighted Young's inequality gives
\begin{align*}
\frac{C_{15}(\sigma)}{2} \color{black} &\leq \int_{-1}^1 \big (1+(\sigma \partial_z u)^2 \big)^{3/4+1/4} \vert \partial_r \psi_u (z,u+1) \vert\, \mathrm{d} z \\
&\leq \frac{1}{2\varepsilon} \int_{-1}^1 \big (1+(\sigma \partial_z u)^2 \big)^{3/2} \vert \partial_r \psi_u(z,u+1) \vert^2 \, \mathrm{d} z + \frac{\varepsilon}{2} \int_{-1}^1 \sqrt{1+(\sigma \partial_z u)^2} \, \mathrm{d} z 
\end{align*}
for $\varepsilon >0$, and multiplying this inequality by $2 \varepsilon$ yields
\begin{align*}
\int_{-1}^1 \big ( 1+ (\sigma \partial_z u)^2 \big)^{3/2} \vert \partial_r \psi_u(z,u+1) \vert^2 \,\mathrm{d} z \geq \varepsilon\, C_{15}(\sigma)  - \varepsilon^2 \int_{-1}^1 \sqrt{1+(\sigma\partial_z u)^2}\, \mathrm{d} z
\end{align*}
as claimed.
\end{proof}

Finally, the last auxiliary result compares the integral of $\arctan(\sigma \partial_z u) \,\sigma\partial_z u$ with the $L_1$-norm of $\partial_z u$:

\begin{lem}\label{Arctan}
For each $t \in [0,T_{max})$, the estimate
\begin{align*}
\int_{-1}^1 \arctan (\sigma \partial_z u  )\, \sigma\partial_z u \, \mathrm{d} z \geq \frac{\pi}{4} \int_{-1}^1  \sqrt{1+(\sigma\partial_z u)^2}\,\mathrm{d} z - \pi 
\end{align*}
holds.
\end{lem}

\begin{proof}
We recall that
\begin{align*}
\arctan(x)x \geq 0 \,, \qquad \arctan(1)=\frac{\pi}{4}\,, \qquad \sqrt{x^2+y^2} \leq \vert x \vert + \vert y \vert
\end{align*}
for $x,y \in \mathbb{R}$ and introduce the set
\begin{align*}
A:= \Big \lbrace z \in [-1,1] \,\Big \vert \, \big \vert \arctan  (\sigma \partial_z u(z)  ) \big \vert \geq \frac{\pi}{4} \, \Big \rbrace \,.
\end{align*}
Noting that $\sigma \vert \partial_z u \vert \leq 1$ on $A^c$, we estimate
\begin{align*}
\int_{-1}^1 \arctan (\sigma \partial_z u  )\, \sigma\partial_z u \, \mathrm{d} z &\geq \int_A \arctan (\sigma \partial_z u  )\, \sigma\partial_z u \, \mathrm{d} z  \\
&\geq \frac{\pi}{4} \int_{A} \sigma \vert \partial_z u \vert \, \mathrm{d} z + \frac{\pi}{4} \int_{A^c} \sigma \vert \partial_z u \vert \, \mathrm{d} z - \frac{\pi}{2} \\
&= \frac{\pi}{4}\int_{-1}^1 \sigma \vert \partial_z u \vert \, \mathrm{d} z + \frac{\pi}{4} \int_{-1}^1 1 \, \mathrm{d} z - \pi \\
&\geq \frac{\pi}{4}  \int_{-1}^1  \sqrt{1+(\sigma\partial_z u)^2}\, \mathrm{d} z - \pi\,.
\end{align*}
\vspace{-0.5cm}
\end{proof}

Based on Lemma \ref{Energyequality}$-$Lemma \ref{Arctan}, we can prove the main result of this section:\\

{\bf Proof of Theorem \ref{MainResultGlobalNonexistence}.}
Let $\lambda >0$ and $(u,\psi_u)$ be a solution to \eqref{filmdimensionless2}-\eqref{RBPsi42} with $u_0 \geq u_{cat}$. We have to show that $T_{max} < \infty$. Since the functional 
$$\mathcal{E}(t)= -\int_{-1}^1 \ln \big (u(t)+1 \big ) \, \mathrm{d} z \,, \quad t \in [0,T_{max})$$ 
is bounded from below by $-2\ln(2)$, it suffices to show 
\begin{align*}
\frac{\mathrm{d}}{\mathrm{d} t} \mathcal{E}(t)<-C < 0 \,, \quad t \in [0,T_{max})\,,
\end{align*}
for some $C >0$ independent of $t$, to exclude the possibility of global existence. Introducing the constant 
$$C_{16}(\sigma)\color{black}:= \frac{1}{\min_{z\in [-1,1]} u_{cat}+1} \in (0,\infty)\,,$$
we note that 
\begin{align*}
\frac{1}{2}\leq \frac{1}{u+1} \leq   C_{16}(\sigma)\color{black}\,,
\end{align*}
as $u$ always stays above the catenoid $u_{cat}$ by Proposition \ref{PCP}. Using \eqref{eqGN1} and \eqref{filmdimensionless2}, we find
\begin{align*}
\frac{\mathrm{d}}{\mathrm{d} t} \mathcal{E}(t) &=-\sigma \int_{-1}^1 \partial_z \mathrm{arctan}(\sigma \partial_z u) \frac{1}{u+1} \, \mathrm{d} z + \int_{-1}^1 \frac{1}{(u+1)^2} \, \mathrm{d} z \\
&\ \ \ - \lambda \int_{-1}^1 \frac{1}{u+1} \big (1+ (\sigma \partial_zu)^2 \big)^{3/2} \vert \partial_r \psi_u(z,u+1) \vert^2\,\mathrm{d} z \\
&\leq - \int_{-1}^1 \frac{\arctan(\sigma \partial_z u)\,\sigma \partial_z u}{(u+1)^2}\, \mathrm{d} z-\bigg [ \frac{\sigma \arctan(\sigma\partial_z u)}{u+1} \bigg ]_{-1}^{1} + 2 \,C_{16}(\sigma)\color{black}^2\\
&\ \ \ - \frac{\lambda}{2} \int_{-1}^1  \big (1+ (\sigma \partial_zu)^2 \big)^{3/2} \vert \partial_r \psi_u(z,u+1) \vert^2\,\mathrm{d} z \\
&\leq - \frac{1}{4} \int_{-1}^1 \arctan(\sigma\partial_z u) \sigma \partial_z u \, \mathrm{d} z + \sigma \pi + 2\,C_{16}(\sigma)\color{black}^2\\
&\ \ \ - \frac{\lambda}{2} \int_{-1}^1  \big (1+ (\sigma \partial_zu)^2 \big)^{3/2} \vert \partial_r \psi_u(z,u+1) \vert^2\,\mathrm{d} z \,.
\end{align*}
Next, for $\varepsilon >0$, Proposition \ref{lowerEstimateRHS} and Lemma \ref{Arctan} imply that
\begin{align*}
\frac{\mathrm{d}}{\mathrm{d} t} \mathcal{E}(t) &\leq -\frac{\pi}{16} \int_{-1}^1 \sqrt{1+(\sigma \partial_z u)^2}\, \mathrm{d} z +\frac{\pi}{4}+\sigma \pi + 2  C_{16}(\sigma)\color{black}^2 \\
&\ \ \ -\frac{\lambda}{2} \bigg ( \varepsilon  C_{15}(\sigma) \color{black} - \varepsilon^2 \int_{-1}^1 \sqrt{1+(\sigma \partial_z u)^2}\,\mathrm{d}z \bigg)\,.
\end{align*}
Choosing $\varepsilon = \displaystyle\sqrt{\frac{\pi}{8\lambda}}$, we reduce this inequality to 
\begin{align}
\frac{\mathrm{d}}{\mathrm{d} t} \mathcal{E}(t) &\leq \frac{\pi}{4}  \sigma \pi - 2  C_{16}(\sigma) \color{black}^2 - \frac{\sqrt{\lambda\,\pi\,} C_{15}(\sigma)}{4 \sqrt{2}}\,.  \label{eqGN4}
\end{align}
The right-hand side is strictly less than zero if $\lambda > \lambda_{crit}(\sigma)$ where 
\begin{align}
\lambda_{crit}(\sigma):= \frac{32}{\pi C_{15}(\sigma)^2} \,\Big ( \frac{\pi}{4}+ \sigma \pi +2 C_{16}(\sigma)^2 \Big )^2\,.  \label{eqGN5}
\end{align}
Hence, for $\lambda > \lambda_{crit}(\sigma)$, the solution $(u,\psi_u)$ cannot be global.\\ \qed

\begin{bem}
{\bf(i)} Computing the smallest possible value for $\lambda_{crit}(\sigma)$ in Theorem \ref{MainResultGlobalNonexistence} is of particular interest.
An upper bound for $\lambda_{crit}(\sigma)$ is given by formula \eqref{eqGN5}, where $\sigma$, the radius of the rings divided by their distance, is easy to determine, and the constants
\begin{align*}
 C_{15}(\sigma)= \int_{-1}^1 (u_{cat}+1)\big ( 1&+ (\sigma \partial_z u_{cat})^2 \big) \partial_r \psi_{cat} (z,u_{cat}+1) \,\mathrm{d} z 
\end{align*}
from \eqref{eqGN3}, and 
\begin{align*}
 C_{16}(\sigma)= \frac{1}{\min_{z \in [-1,1]} u_{cat}+1} \,, \qquad  \min_{z \in [-1,1]} u_{cat}= \frac{1}{\mathrm{cosh}(c)} -1\,
\end{align*}

may be accessible through numerical computations. \color{black}\\
{\bf(ii)} A consequence of the proof of Theorem \ref{MainResultGlobalNonexistence} is that, for given $\lambda > \lambda_{crit}(\sigma)$, there exists a uniform upper bound on the blow-up time $T_{max}$\,: Abbreviating the right-hand side of \eqref{eqGN4} by
\begin{align*}
-C_{17}(\sigma,\lambda):= \frac{\pi}{4}  \sigma \pi - 2  C_{16}(\sigma) \color{black}^2 - \frac{\sqrt{\lambda\,\pi\,} C_{15}(\sigma)}{4 \sqrt{2}}\,,
\end{align*}
we deduce from \eqref{eqGN4}, the fact that $u_0 \geq u_{cat}$, and the definition of $\mathcal{E}$ that
\begin{align*}
 \mathcal{E}(t) &= \mathcal{E}(0) + \int_0^t \frac{\mathrm{d}}{\mathrm{d} \tau} \mathcal{E}(\tau) \, \mathrm{d} \tau \\
 &\leq - \int_{-1}^1 \ln\big(u_{cat}(z)+1 \big )\,\mathrm{d} z - t \,C_{17}(\sigma,\lambda) \,, \qquad t \in [0,T_{max})\,.
\end{align*}
Now, using $\mathcal{E}(t) \geq -2 \ln(2)$, we find
\begin{align*}
 T_{max} \leq \left (2\ln(2)- \int_{-1}^1 \ln \big (u_{cat}(z)+1 \big )\,\mathrm{d} z \right ) \,C_{17}(\sigma,\lambda)^{-1}\,,
\end{align*}
where the right-hand side is independent of the initial value $u_0 \geq u_{cat}$. 
\end{bem}

\section*{Acknowledgement} This paper contains results and edited text from my PhD-thesis. I am very thankful to my PhD-supervisor Christoph Walker for suggesting the research topic and for his valuable advice.

\section{Appendix: Elliptic Regularity on Convex Domains}\label{ERCD}

The goal of this appendix is to give a new and detailed proof of an elliptic regularity result on bounded, convex domains from \cite[Theorem 3.10.1]{LU68} which has been used in Lemma \ref{selliptic3}: 

\begin{thm}\label{5}
Suppose that 
 \begin{itemize}
 \item $\Omega_0 \subset \mathbb{R}^n$ is convex, open and bounded,
 \item $A \in \big [ W^1_q(\Omega_0 ) \big ]^{n \times n}$, where $q >n$, is symmetric,
 \item there is $\alpha > 0$ with
 \begin{align*}
  \xi^T A(x) \xi \geq \alpha \vert \xi \vert^2 \,, \qquad x \in \overline{\Omega}_0 \,, \qquad \xi \in \mathbb{R}^n \,.
 \end{align*}
\end{itemize}
Then, for each $F \in L_2(\Omega_0)$, the problem 
\begin{align}
\begin{cases}
 \ -\mathrm{div} \big ( A(x)\nabla \phi \big)&= F \quad \text{in} \quad \Omega_0 \,, \\ 
 \ \ \qquad \qquad \quad \phi &=0 \quad \text{on} \quad \partial \Omega_0 \label{4.eq0}
  \end{cases}
\end{align}
has a unique solution $\phi \in W^2_{2,D}(\Omega_0)$. Moreover, there exists a constant $C$ depending only on $q$, $n$, $\Omega_0$, the $W^1_q$-norm of the coefficients of $A$ and the ellipticity constant $\alpha$ such that 
\begin{align}
 \Vert \phi \Vert_{W_2^2(\Omega_0)} \leq C \, \Vert F \Vert_{L_2(\Omega_0)} \,. \label{4.eq0b} 
\end{align}
\end{thm}

We note that, thanks to the Riesz Representation Theorem, problem \eqref{4.eq0} has a unique weak solution $\phi \in W^1_{2,D}(\Omega_0)$. 

The difficult part is to improve its regularity because such an improvement of regularity is usually derived for bounded $C^2$-domains, see for example \cite[Theorem 6.3.4]{Evans10}. 
But convex domains as assumed in Theorem \ref{5} are merely Lipschitz domains \cite[Corollary 1.2.2.3]{Grisvard85}, for which, in general, an improvement of regularity may even fail. \\

In \cite[Theorem 3.2.1.2]{Grisvard85} the special case $q=\infty$ of Theorem \ref{5} is proven. The argument is based on an approximation of $\Omega_0$ from the inside by a sequence of smooth convex domains $(\Omega_m)$ on which the elliptic problem possesses a sequence of unique $W^2_2$-solutions $(\phi_m)$. 
Furthermore, using the convexity of the domains, a remarkable $W^2_2$-a-priori estimate for $\phi_m$ independent of $m$ is derived, which allows the author to extract a subsequence of $(\phi_m)$ converging to a $W^2_2$- solution to the original problem on $\Omega_0$. We derive an improved $W^2_2$-a-priori-estimate to combine the idea of domain approximation from \cite{Grisvard85} with an approximation of the coefficients of $A(\, \cdot \,)$ to reduce the assumption from $q=\infty$ to $q > n$ and thus to give a new proof of Theorem \ref{5}. \\

An important ingredient in our proof is Sobolev's embedding theorem on convex domains as it allows to characterize the embedding constant solely by the volume and diameter of the domain. 

\begin{prop}\label{-1}
Let $\Omega$ be a convex, bounded and open subset of $\mathbb{R}^n$, let $q >n$. Then, there exists a constant $C >0$ depending only on $q$, $n$, the diameter $\mathrm{diam}(\Omega)$ and the volume $\vert \Omega \vert$ such that
\begin{align*}
 \Vert v \Vert_{L_{\frac{2q}{q-2}}(\Omega)} \leq C \, \Vert v \Vert_{W^1_2(\Omega)} \,, \qquad v \in W^1_2(\Omega)\,.
\end{align*}
\end{prop}

\begin{proof}
This is a special case of \cite{MTSO17}.
\end{proof}

A direct consequence of Sobolev's embedding theorem on convex domains is the following: 

\begin{cor}\label{0}
Let $\Omega$ be a convex, bounded and open subset of $\mathbb{R}^n$. Let $q > n$ and $\delta > 0$. Then, there exists a constant $C >0$ depending only on $q, n,$ \color{black} $\delta$, $\mathrm{diam}(\Omega)$ and $\vert\Omega\vert$ such that 
\begin{align}
\Vert v \Vert_{L_{\frac{2q}{q-2}}(\Omega)}^2 \leq \delta \Vert v \Vert_{W_2^1(\Omega)}^2 +C \,\Vert v \Vert_{L_2(\Omega)}^2 \label{0.eq1}
\end{align}
for all $v \in W^1_2(\Omega)$.
\end{cor}

\begin{proof}
 For fixed $\varepsilon \in (0, q-n)$, we deduce from 
\begin{align*}
 2 < \frac{2q}{q-2} < \frac{2(q-\varepsilon)}{(q- \varepsilon) -2}
\end{align*}
and Hölder's inequality that
\begin{align*}
 \Vert v \Vert^2_{L_{\frac{2q}{q-2}}(\Omega)} \leq \Vert v \Vert^{2(1-\theta)}_{L_2(\Omega)}\, \Vert v \Vert^{2\theta}_{L_{\frac{2(q-\varepsilon)}{(q-\varepsilon)-2}}(\Omega)} 
\end{align*}
for suitable $\theta =\theta(q,n)\in (0,1)$ and all $v \in W^1_2(\Omega)$. Now an application of Proposition~\ref{-1} together with the weighted Young's inequality completes the proof of \eqref{0.eq1}.
\end{proof}

We can now turn to the improved a-priori estimate on smooth convex domains (here that means: $\partial \Omega \in C^2$). The starting point is the following result from \cite{Grisvard85}:

\begin{prop}\label{2}
Let $\Omega$ be a convex, bounded open subset of $\mathbb{R}^n$ with a $C^2$-boundary, and $A$ a symmetric $n \times n$-matrix with each eigenvalue larger than $\alpha >0$. Then, there exists a constant $C_{18}$ depending only on the diameter of $\Omega$ and $\alpha$ such that 
\begin{align}
 \Vert \phi \Vert_{W^2_2(\Omega)} \leq C_{18} \, \big \Vert \mathrm{div} \big (A \nabla  \phi \big ) \big \Vert_{L_2(\Omega)} \,, \qquad \phi \in W^2_2(\Omega ) \cap W^1_{2,D}(\Omega) \,. \label{2.eq1}
\end{align}
\end{prop}

\begin{proof} 
This follows from the first step in the proof of \cite[Lemma 3.1.3.2]{Grisvard85} by carefully analysing the appearing constants.
\end{proof}

Next, we aim at extending the improved a-priori estimate \eqref{2.eq1} to elliptic operators with variable coefficients. Our result is a generalization of the a-priori estimate in~\cite[Lemma 3.1.3.2, Theorem 3.1.3.1]{Grisvard85}. As in \cite[Lemma 3.1.3.2]{Grisvard85}, we treat variable coefficient operators locally as a perturbation of constant coefficient operators.
Our new ingredient is the Sobolev's embedding theorem for convex domains. It allows us to formulate an a-priori estimate in which the constant does not depend on the $W^1_\infty$-norm of the coefficient matrix $A(\,\cdot\,)$ as in \cite{Grisvard85}, but only on the $W^1_q$-norm of this matrix for $q >n$.

\begin{prop}\label{3}
 Suppose that 
 \begin{itemize}
 \item $\Omega_0 \subset \mathbb{R}^n$ is convex, open and bounded,
 \item $q > n$,
 \item $A\in \big [ C^\infty( \overline{\Omega}_0 ) \big ]^{n \times n}$ is symmetric,
 \item there is $\alpha > 0$ with
 \begin{align*}
  \xi^T A(x) \xi \geq \alpha \vert \xi \vert^2 \,, \qquad x \in \overline{\Omega}_0 \,, \qquad \xi \in \mathbb{R}^n \,.
 \end{align*}
\end{itemize}
Then, there exists a constant $C$ depending only on the $W^1_q$-norm of the coefficients of $A$, the ellipticity constant $\alpha$ and $\Omega_0$ such that for each convex and open $\Omega \subset \Omega_0$ with $C^2$-boundary and $\vert \Omega \vert \geq \frac{1}{2} \vert \Omega_0 \vert$ the estimate
 \begin{align}
  \Vert \phi\Vert_{W_2^2(\Omega)} \leq C \, \big \Vert \mathrm{div} \big (A(\,\cdot\,) \nabla \phi \big  ) \big \Vert_{L_2(\Omega)}\,,  \qquad \phi \in W^2_2(\Omega ) \cap W^1_{2,D}(\Omega)\,, \label{3.eq0}
 \end{align}
 holds.
\end{prop}

\begin{proof} {\bf(i)} {\it Local Estimate:} \\
Near fixed $x_0 \in \overline{\Omega}_0$, our first goal is to prove a local version of \eqref{3.eq0}. To this end, we interpret the operator locally as a perturbation of the constant coefficient operator
 $-\mathrm{div} \big (A(x_0)\nabla \phi \big )$ treated in Proposition~\ref{2}. Assume that $ \phi \in W^1_{2,D}(\Omega)$ with support contained in $\mathbb{B}(x_0,r) \cap \Omega$ 
 (where $r >0$ will be determined later). Writing $A(x_0)= [a_{ij}(x_0)]$ and $A(x) = [a_{ij}(x)]$ we deduce from
  \begin{align*}
  \mathrm{div} \big (A(x_0)\nabla \phi \big )&-\mathrm{div} \big (A(x)\nabla \phi \big ) \\
  &= \sum_{i,j=1}^n \big ( a_{ij}(x_0)-a_{ij}(x) \big ) \partial_i \partial_j \phi - \sum_{i,j=1}^n \partial_i a_{ij}(x) \partial_j \phi
 \end{align*}
 that 
 \begin{align*}
 \big \vert \mathrm{div} \big (A(x_0)\nabla \phi \big )&-\mathrm{div} \big (A(x)\nabla \phi \big )\big \vert^2 \\
 &\leq C (n) \Big ( \sum_{i,j=1}^n \vert a_{ij}(x_0) - a_{ij}(x) \vert^2 \, \vert \partial_{i}\partial_j \phi \vert^2 + \sum_{i,j=1}^n \vert \partial_i a_{ij}(x) \vert^2 \, \vert \partial_j \phi \vert^2 \Big )\,.
 \end{align*}
Integrating with respect to $x\in\Omega$ and using that $W^1_q(\Omega_0) \hookrightarrow C^s(\overline{\Omega}_0)$ with $s = 1- n/q>0$, we get
\begin{align}
 \big \Vert \mathrm{div} \big (A(x_0)\nabla \phi \big )&-\mathrm{div} \big (A(\,\cdot\,)\nabla \phi \big )\big \Vert_{L_2(\Omega)}^2 \notag \\
 &\leq C \big ( n, \Vert A \Vert_{W^1_q(\Omega_0)} ,  \Omega_0 \color{black} \big ) \bigg ( r^{2s} \, \Vert \phi \Vert^2_{W_2^2(\Omega)}  + \sum_{i,j=1}^n \int_{\Omega} \vert \partial_i a_{ij} \vert^2 \, \vert \partial_j \phi \vert^2 \, \mathrm{d} x \bigg )\,. \label{3.eq1}
\end{align}
Applying Hölder's inequality with exponents $\frac{2}{q} + \frac{q-2}{q}=1$ together with Corollary~\ref{0} to the second term in \eqref{3.eq1} gives
\begin{align}
 \int_{\Omega} \vert \partial_i a_{ij} \vert^2 \, \vert \partial_j \phi \vert^2 \, \mathrm{d} x &\leq \Vert a_{ij} \Vert_{W^1_q(\Omega)}^2 \,\Vert \partial_j \phi\Vert_{L_\frac{2q}{q-2}(\Omega)}^2\notag \\
 &\leq  C \big (\Vert A \Vert_{W^1_q(\Omega_0)}, q, n , \Omega_0 \big) \Big ( \delta \,\Vert \partial_j \phi \Vert_{W_2^1(\Omega)}^2 + C (\delta)  \,\Vert \partial_j \phi \Vert^2_{L_2(\Omega)} \Big ) \notag
\end{align}
for each $\delta > 0$. Plugging this back into \eqref{3.eq1} yields
\begin{align}
\big \Vert \mathrm{div} \big (A(x_0)\nabla \phi\big )&-\mathrm{div} \big (A(\,\cdot\,)\nabla \phi \big )\big \Vert_{L_2(\Omega)}^2 \notag \\
&\leq  C \big ( \Vert A \Vert_{W^1_q(\Omega_0)}  , n,q , \Omega_0\big) \color{black} \Big ( (r^{2s}+ \delta) \, \Vert \phi \Vert^2_{W_2^2(\Omega)} 
+ C (\delta) \Vert \phi \Vert^2_{W^1_2(\Omega)} \Big )\,.\label{3.eq3}
\end{align}
Now we infer from Proposition~\ref{2}, the triangle inequality and \eqref{3.eq3} that there exists a constant $C_{19}= C_{19} \big ( \Vert A \Vert_{W^1_q(\Omega_0)},q,n,\Omega_0 \big )$ with
\begin{align*}
 \Vert \phi \Vert^2_{W^2_2(\Omega)} 
 &\leq 2  \, C_{18}^2 \,\color{black} \Big ( \big \Vert \mathrm{div} \big (A(\, \cdot \,) \nabla \phi \big ) \big \Vert_{L_2(\Omega)}^2+\big \Vert \mathrm{div} \big (A(x_0) \nabla \phi \big )-\mathrm{div} \big (A(\, \cdot \,) \nabla \phi \big ) \big \Vert_{L_2(\Omega)}^2 \Big )\\
 &\leq C_{19}\,\Big ( \big \Vert \mathrm{div} \big (A(\, \cdot \,) \nabla \phi \big ) \big \Vert_{L_2(\Omega)}^2+(r^{2s}+ \delta) \, \Vert \phi \Vert^2_{W_2^2(\Omega)} 
+ C (\delta) \Vert \phi \Vert^2_{W^1_2(\Omega)} \Big )\,.
\end{align*}
Choosing $r$ and $\delta >0$ with
$$( r^{2s} +\delta ) \leq  \frac{1}{2\,  C_{19} }\,, $$
we arrive at
\begin{align}
 \Vert \phi \Vert^2_{W^2_2(\Omega)} &\leq C\big (\Vert A \Vert_{W^1_q(\Omega_0)}, q,n, \Omega_0 \big )\, \Big (  \big \Vert \mathrm{div} \big (A(\,\cdot\,)\nabla \phi \big ) \big \Vert_{L_2(\Omega)}^2  + \Vert \phi \Vert_{W_2^1(\Omega)}^2 \Big ) \label{3.eq4}
\end{align}
for all $\phi \in W^2_{2,D}(\Omega)$ with support contained in $\mathbb{B}(x_0,r) \cap \Omega$.
Finally, it follows from Friedrich's inequality, Gauss's theorem and the weighted Young's inequality that
\begin{align*}
 \Vert \phi\Vert^2_{W^1_2(\Omega)} &\leq C \big ( \mathrm{diam}(\Omega_0) \big ) \Vert \nabla \phi \Vert^2_{L_2(\Omega)} \\
 &\leq \tfrac{1}{\alpha}\, C \big ( \mathrm{diam}(\Omega_0) \big ) \int_\Omega \nabla \phi^T A(x) \nabla \phi \, \mathrm{d} x \\
 &\leq  C \big ( \mathrm{diam}(\Omega_0), \alpha \big ) \,\big \Vert \mathrm{div}\big(A(\,\cdot\,)\nabla \phi\big) \Vert^2_{L_2(\Omega)}+ \tfrac{1}{2} \Vert \phi \Vert_{W_2^1(\Omega)}^2\,,
\end{align*}
so that we can eliminate the $W^1_2$-norm of $\phi$ on the right-hand side of \eqref{3.eq4} and get
\begin{align}
 \Vert \phi \Vert^2_{W^2_2(\Omega)} \leq C \,  \big \Vert \mathrm{div} \big (A(\,\cdot\,)\nabla \phi  \big ) \big \Vert_{L_2(\Omega)}^2  \label{3.eq5}
\end{align}
for all $\phi  \in W^1_{2,D}(\Omega)$ with support contained in $\mathbb{B}(x_0,r) \cap \Omega$ where $C$ and also $r$ depend only on $q$, $n$ \color{black}, $\Vert A \Vert_{W^1_q(\Omega_0)}$, the ellipticity constant $\alpha$, and $\Omega_0$ (but not on $\Omega$).\\

{\bf(ii)} {\it Global Estimate:}\\
We aim for a global version of \eqref{3.eq5}. To this end, note that $\overline{\Omega}_0$ is compact so that there are $x_1, \dots, x_m \in \overline{\Omega}_0$ with $$\overline{\Omega}_0 \subset \bigcup_{i=1}^m \mathbb{B}(x_i,r).$$ 
Let $\lbrace \theta_i \, \vert \, i=1, \dots, m \rbrace$ be a smooth partition of unity on $\overline{\Omega}_0$ subordinated to $\bigcup_{i=1}^m \mathbb{B}(x_i,r)$. For $\phi  \in W^2_{2,D}(\Omega) $,
it follows that $\theta_i \phi  \in W^2_{2,D}(\Omega)$ with support in $\mathbb{B}(x_i,r) \cap \Omega$ and
\begin{align}
 \Vert \phi \Vert_{W^2_2(\Omega)} 
 &\leq C \big (\Vert A \Vert_{W^1_q(\Omega_0)} , n,q,\Omega_0,\alpha \big )\,\sum_{i=1}^m \big \Vert \mathrm{div} \big ( A(\,\cdot\,) \nabla (\theta_i \phi  ) \big ) \big \Vert_{L_2(\Omega)}\label{3.eq6}
\end{align}
thanks to \eqref{3.eq5}. For the right-hand side, we compute 
\begin{align*}
 \mathrm{div} \big ( A(x) \nabla (\theta_i \phi  ) \big ) &=\theta_i \,\mathrm{div} \big (A(x) \nabla \phi  \big ) +\sum_{j=1}^n \Big ( (\partial_j \phi )\big [A(x) \nabla \theta_i \big ]_j + (\partial_j \theta_i) \big [A(x) \nabla \phi  \big ]_j \Big )\\
 &+ \phi \,\mathrm{div} \big ( A(x) \nabla \theta_i \big ) \\
&=:I+II+III\,.
\end{align*}
Since
\begin{align*}
 \Vert \theta_i \Vert_{C^2(\overline{\Omega}_0)} \leq C \big ( \Vert A \Vert_{W^1_q(\Omega_0)} , q,n,\Omega_0,\alpha \big)
\end{align*}
for each $i=1, \dots,m$, we find 
\begin{align}
 \Vert I \Vert_{L_2(\Omega)} \leq C \big ( \Vert A \Vert_{W^1_q(\Omega_0)} , q,n,\Omega_0,\alpha \big) \, \big \Vert \mathrm{div} \big (A(\,\cdot\,) \nabla \phi \big ) \big \Vert_{L_2(\Omega)} \label{3.eq7}
\end{align}
as well as
\begin{align}
 \Vert II \Vert_{L_2(\Omega)} &\leq 2n^2 \Vert \theta_i \Vert_{C^1(\overline{\Omega}_0)} \, \Vert A \Vert_{C (\overline{\Omega}_0)} \, \Vert \nabla \phi \Vert_{L_2(\Omega)}\notag \\
 &\leq C \big ( \Vert A \Vert_{W^1_q(\Omega_0)} , q,n,\Omega_0,\alpha \big) \, \Vert \nabla \phi \Vert_{L_2(\Omega)} \,,
\end{align} 
where we used that $W^1_q(\Omega_0) \hookrightarrow C ( \overline{\Omega}_0)$. For $III$, we compute further
$$III= \sum_{j,k=1}^n \phi \, \partial_j a_{jk} \,\partial_k \theta_i + a_{jk}\, \phi \,\partial_j\partial_k \theta_i\,,$$
and hence
\begin{align*}
 \Vert III \Vert_{L_2(\Omega)} \leq C \big ( \Vert A \Vert_{W^1_q(\Omega_0)} , q,n,\Omega_0,\alpha \big )  \, \sum_{j,k=1}^n \Big (\Vert \phi \,\partial_j a_{jk} \Vert_{L_2(\Omega)} + \Vert \phi  \Vert_{L_2(\Omega)} \Big )\,,
\end{align*}
where we applied the embedding $W^1_q(\Omega_0) \hookrightarrow C (\overline{\Omega}_0)$ once more. Since
\begin{align*}
 \Vert \phi \,\partial_j a_{jk} \Vert_{L_2(\Omega)}^2 = \int_\Omega \phi ^2 (\partial_j a_{jk})^2 \, \mathrm{d} x &\leq \, \Vert \phi  \Vert_{L_{\frac{2q}{q-2}}(\Omega)}^2 \Vert A \Vert_{W^1_q(\Omega_0)}^2 \\
 &\leq C \big ( \Vert A \Vert_{W^1_q(\Omega_0)} , q,n,\Omega_0\big ) \Vert \phi \Vert_{W^1_2(\Omega)}^2
\end{align*}
due to Hölder's inequality and Sobolev's embedding theorem~\ref{-1}, we arrive at 
\begin{align}
 \Vert III \Vert_{L_2(\Omega)} \leq C \big ( \Vert A \Vert_{W^1_q(\Omega_0)} , q,n,\Omega_0,\alpha \big ) \Vert \phi  \Vert_{W^1_2(\Omega)}\,. \label{3.eq8}
\end{align}
Plugging the estimates for $I$ to $III$ in \eqref{3.eq7}-\eqref{3.eq8} back into \eqref{3.eq6}, we find that 
\begin{align*}
 \Vert \phi  \Vert_{W^2_2(\Omega)} \leq C \big ( \Vert A \Vert_{W^1_q(\Omega_0)} , q,n,\Omega_0,\alpha \big )  \, \Big (  \big \Vert \mathrm{div} \big (A(\,\cdot\,) \nabla \phi  \big ) \big \Vert_{L_2(\Omega)} + \Vert \phi  \Vert_{W_2^1(\Omega)} \Big ) 
\end{align*}
for $\phi \in  W^2_{2,D}(\Omega)$. Eventually, we can apply the same steps which lead to \eqref{3.eq5} to eliminate the $W^1_2$-norm of $\phi$ on the right-hand side.
\end{proof}

Based on the improved a-priori estimate \eqref{3.eq0}, we can complete the proof of Theorem~\ref{5}\,. \\

{\bf Proof of Theorem \ref{5}.}
{\bf(i)} {\it Approximation of the domain:} \\
Suppose first that $A \in \big [C^\infty(\overline{\Omega}_0)\big ]^{n\times n}$. Then, we can follow the lines of the proof of \cite[Theorem 3.2.1.2]{Grisvard85} with \cite[Equation (3.2.1.3)]{Grisvard85} replaced by the improved a-priori estimate from Proposition~\ref{3}. Hence, problem \eqref{4.eq0} has a unique solution $\phi  \in W^2_{2,D}(\Omega_0)$ 
which additionially satisfies the estimate 

\begin{align}
 \Vert \phi  \Vert_{W^2_2(\Omega_0)} \leq C \big ( q,n, \Omega_0, \Vert A \Vert_{W^1_q(\Omega)} , \alpha \big )\,  \Vert F \Vert_{L_2(\Omega_0)} \,. \label{4.eq1} \\ \notag
\end{align}

{\bf(ii)} {\it Approximation of the coefficients:}\\
Now we treat the general case $A \in \big [W^1_q(\Omega_0) \big ]^{n \times n}$ with $q > n$. Recall that $\Omega_0$ has a Lipschitz boundary so that we find a sequence $(A^{(m)})  \subset \big [ C^\infty(\overline{\Omega}_0) \big ]^{n \times n}$ 
such that each $A^{(m)}(x)$ is symmetric and $A^{(m)} \rightarrow A$ in $\big [W^1_q(\Omega_0) \big ]^{n \times n}$. Moreover, we may assume that 
$$\sup_m \Vert A^{(m)} \Vert_{ W^1_q(\Omega_0)} \leq 2 \Vert A \Vert_{W^1_q(\Omega_0)}\,.$$
It remains to arrange that the $(A^{(m)})$ have a common ellipticity constant. To this end, note that $q >n$ implies that 
\begin{align*}
 \big \vert \xi^T A(x) \xi - \xi^T A^{(m)}(x) \xi \big \vert &\leq \vert \xi^T \vert \, \Vert A-A^{(m)} \Vert_{C (\overline{\Omega}_0)} \vert \xi \vert \\
 &\leq C \, \Vert A-A^{(m)} \Vert_{W^1_q(\Omega_0)} \rightarrow 0
\end{align*}
for $m \rightarrow \infty$ and for each $\xi \in \mathbb{R}^n$ with $\vert \xi \vert =1$ and $x \in \overline{\Omega}_0$. Hence, we may assume without loss of generality that
\begin{align*}
 \big \vert \xi^T A(x) \xi - \xi^T A^{(m)}(x) \xi \big \vert \leq \alpha/2  \,, \qquad \vert \xi \vert =1\,, \ x \in \overline{\Omega}_0 \,, \ m \in \mathbb{N}\,,
\end{align*}
which immediately implies that each $A^{(m)}$ is uniformly elliptic with a common ellipticity constant $\alpha/2$.

Now it follows from part (i) that there exists a unique solution $\phi_m \in W^2_{2,D}(\Omega_0)$ to the problem
\begin{align*}
\begin{cases}
 \ -\mathrm{div} \big ( A^{(m)}(x)\nabla \phi_m \big)&= F \quad \text{in} \quad \Omega_0 \,, \\ 
  \qquad \qquad \qquad \quad \phi_m &=0 \quad \text{on} \quad \partial \Omega_0
  \end{cases}
\end{align*}
with 
\begin{align}
 \Vert \phi_m \Vert_{W^2_2(\Omega_0)} \leq C \big ( q,n, \Omega_0, \Vert A \Vert_{W^1_q(\Omega)} , \alpha \big )\,  \Vert F \Vert_{L_2(\Omega_0)}\,, \qquad m \in \mathbb{N} \,, \label{4.eq2}
\end{align}
due to \eqref{4.eq1}. Hence, we find a subsequence $(\phi_m)$ and $\phi  \in W^2_{2,D}(\Omega_0)$ with $\phi_m \rightarrow \phi$ in $W^1_2(\Omega_0)$ and $\phi_m \rightharpoonup \phi$ in $W^2_2(\Omega_0)$. Letting $m \rightarrow \infty$ in the weak formulation
\begin{align*}
  \int_{\Omega_0} \nabla \phi_m^T A^{(m)}(x) \nabla \varphi \, \mathrm{d} x = \int_{\Omega_0} F \varphi \,\mathrm{d} x \,, \qquad \varphi \in \mathcal{D}(\Omega_0)\,,
\end{align*}
we see that $\phi$ is a solution to \eqref{4.eq0}. It is unique due to the Riesz Representation Theorem. Finally, estimate \eqref{4.eq0b} follows from \eqref{4.eq2} and the weak lower semi-continuity of $\Vert \, \cdot \, \Vert_{W^2_2(\Omega_0)}$.
\qed

\bibliographystyle{siam}
\bibliography{BibliographyDoc}
\end{document}